\documentclass[11pt,reqno]{amsart}
\usepackage{amsmath, amssymb, amsthm}
\usepackage{url}
\usepackage[breaklinks]{hyperref}
\usepackage{hyperref}
\usepackage{amssymb}
\usepackage{xcolor}

\usepackage{a4wide}

\renewcommand{\Im}{\operatorname{Im}}

\renewcommand{\Re}{\operatorname{Re}}
\renewcommand{\Im}{\operatorname{Im}}

\newcommand{\ZZ}{\mathbb{Z}}
\newcommand{\SL}{\operatorname{SL}}

\renewcommand{\(}{\left\(}
\renewcommand{\)}{\right\)}
\renewcommand{\[}{\left\[}
\renewcommand{\]}{\right\]}

\numberwithin{equation}{section}
\theoremstyle{plain}
\newtheorem{theorem}{Theorem}[section]
\newtheorem{lemma}[theorem]{Lemma}
\newtheorem{remark}[]{Remark}
\newtheorem{definition}[theorem]{Definition}

\newtheorem{corollary}[theorem]{Corollary}
\newtheorem{proposition}[theorem]{Proposition}
\newtheorem{example}[]{Example}
\usepackage{amsmath}
\usepackage{mathtools}
\usepackage{enumitem}

  \title{Riesz means of quadratic class numbers}
\author{Olivia Beckwith}
\address{Department of Mathematics, Tulane University, New Orleans, LA 70118}
\email{obeckwith@tulane.edu}

\author{Tushar Karmakar}
\address{Department of Mathematics, Tulane University, New Orleans, LA 70118}
\email{tkarmakar@tulane.edu}

\makeatletter
\@namedef{subjclassname@2020}{%
	\textup{2020} Mathematics Subject Classification}
\makeatother

\begin{document}
	\maketitle

	\begin{abstract}
	We prove an asymptotic formula for a weighted Riesz mean of Hurwitz class numbers and real quadratic class numbers.	 To do this, we introduce $L$-functions for weight $\frac{1}{2}$ sesquiharmonic Maass forms of moderate growth and prove a formula for the Riesz means of the corresponding generalized mock modular forms, generalizing a recent result of the first author with Diamantis, Gupta, Rolen, and Thalagoda for mock modular forms. We then apply this formula to a sesquiharmonic Maass form that was first introduced by Duke, Imamo\={g}lu, and T\'{o}th. 
		\end{abstract}

\section{Introduction}
 The study of the growth of class numbers of quadratic number fields dates back to Gauss' conjectures. In the case of negative discriminants, Gauss Conjectures' were resolved throughout the twentieth century in several celebrated works \cite{GoldfeldGrossZagier, GrZa, Heegner, Heilbronn, Siegel, Stark}, while for positive discriminants, Gauss' conjectures remain open. 
 
 The richness of the study of class numbers owes in part to the connections between class numbers and other important objects in analytic number theory, especially $L$-functions and modular forms. The works of Heilbronn \cite{Heilbronn} and Siegel \cite{Siegel} relied on Dirichlet's class number formula for the value of quadratic Dirichlet $L$-series at 1. The connections to modular forms go back to Gauss, who related class numbers to coefficients of the third power of the Jacobi theta function. More precisely, these coefficients are given in terms of Hurwitz class numbers $H(n)$, which are defined in terms of binary quadratic forms of negative discriminant and have been studied extensively in the literature: their recursive relations such as the Kronecker-Hurwitz relations \cite{CohenClassNumbers}, traces of Hecke operators acting spaces of modular forms \cite{mertens2016eichler}, divisibility properties \cite{beckwith2017indivisibility, beckwith2020class, ankeny1955divisibility, ono1999indivisibility, BRR3}, elliptic curves \cite{Schoof}, and mathematical physics \cite{Manschot}.
 
Many of these works rely the mock modularity of the generating function of $H(n)$, known due to work by Zagier \cite{zagier75}. Unlike the usual modular forms, mock modular forms lack the usual modular symmetry, but instead have nonholomorphic modular completions called harmonic Maass forms, which consist of a sum of a holomorphic form and a nonholomorphic generating series involving the incomplete gamma function. The study of such functions can be traced back to back to Ramanujan's seminal discovery of mock theta functions, but was developed in its modern formulation in the works of Zwegers \cite{Zwegers} and Bruinier and Funke \cite{brufu02}. An overview of the theory of harmonic Maass forms and their applications in number theory, partition statistics, representation theory, and various areas of mathematical physics can be found in \cite{thebook}. 

Positive discriminant analogues $\mathrm{Tr_d(1)}$ of Hurwitz class numbers were introduced by Duke, Imamo\={g}lu and T\'{o}th \cite{DukeImamogluToth} and shown to be coefficients of the holomorphic part of a nonholomorphic modular form belonging to the so-called class of sesquiharmonic Maass forms. These functions have since appeared in the setting of the smallest parts function in partition theory \cite{Bringmann-2008}, non-critical values of modular $L$-functions \cite{BDR}, and theta lifts \cite{BFI}. While the Fourier expansions of harmonic Maass forms involving two infinite generating series (a holomorphic series and a generating series involving the incomplete gamma function), the Fourier series for sesquiharmonic Maass forms include a third generating series involving an additional special function. 

In this paper, we focus on the class of weight $\frac{1}{2}$ sesquiharmonic Maass forms of moderate growth (see the next section for precise definitions). This class includes the function $\hat{Z}(\tau)$ defined in \cite{DukeImamogluToth} and the higher level generalizations introduced in \cite{BeckwithMono-I}. They were used in \cite{allen-beckwith-sharma} to give a half-integral weight Eisenstein analogue of a result of \cite{mertens2016} on shifted convolution sums. Functions in this class have Fourier expansions of the form
\begin{align}\label{eq:fourier}
	F(\tau)&= \left(d_F(0) +d_F(1)\log v+d_F(2)y^{\frac{1}{2}} \right)+ \sum_{\substack{n >0} } {c_F(n) q^{n}} \nonumber \\
&+ \sum_{\substack{n < 0} } {b_F(n) \beta(-4nv)q^n} + \sum_{n \ge 1}{a_F(n)\alpha(4nv)q^n},
 \end{align}
where throughout we use the standard notation $\tau = u + iv \in \mathbb{H}$, $q= e^{2 \pi i \tau}$, and for $y > 0$, the special functions $\alpha(y)$ and $\beta (y)$ are given by
\begin{align}\label{expression-alphabeta}
\alpha{(y)} := \frac{\sqrt{y}}{4\pi} {\int_{0}^\infty t^{-1/2}e^{-{\pi} yt}\log(1+t)dt}, \quad  \beta{(y)} := \frac{1}{\sqrt{\pi}} \int_{{\pi}y}^\infty {t^{1/2 -1} e^{-t}} dt.
\end{align}

A central feature of the classical theory of holomorphic modular forms is the associated $L$-series, whose analytic continuation and functional equation and (for Hecke eigenforms) Euler products are useful in studying the Fourier coefficients and related sums. Nonholomorphic generating series pose an obstruction to these methods. The problem of defining $L$-functions for harmonic Maass forms has been studied recently by \cite{ShanSingh22}, who restricted to the class of harmonic Maass forms of moderate growth, and Diamantis, Lee, Raji, and Rolen \cite{DLRR}, who employed test functions to deal with both the nonholomorphic obstruction and the exponential growth of harmonic Maass forms. Here we extend the method of Shankadhar and Singh and define $L$-functions for sesquiharmonic Maass forms of weight $\frac{1}{2}$ with moderate growth. With $F$ given as in \eqref{eq:fourier}, we define the completed Dirichlet series $\Lambda_{N}(F,s)$ associated to $F$ as follows:
\begin{align}\label{eq:complete-L-function}
\Lambda_{N}(F,s)= \left(\frac{\sqrt{N}}{2 \pi}\right)^{s} \left[\Gamma(s)L^{+}(F,s) + \frac{1}{\sqrt{\pi}} W_{\frac{1}{2}}(s)L^{-}(F,s)+ Z(s)L^{*}(F,s)\right],
\end{align}
where the special functions $W_{\frac{1}{2}}(s)$ and $Z(s)$ are defined in \eqref{eq:W-def} and \eqref{eq:Z-def} in the next section section, and
$$L^{+}(F,s)= \sum_{n=1}^{\infty} \frac{c_{F}(n)}{n^{s}}, \quad L^{-}(F,s)= \sum_{n=1}^{\infty} \frac{b_{F}(-n)}{n^{s}}, \quad L^{*}(F,s)= \sum_{n=1}^{\infty} \frac{a_{F}(n)}{n^{s}}.$$
We show that these series converge absolutely for sufficiently large $Re(s)$.

Our first main result is the meromorphic continuation and functional equation of $\Lambda_N (F,s)$. The residues and functional equation involves the function
$$
G (\tau ) = F |_{\frac{1}{2}} W(N) := N^{\frac{1}{4}} (N \tau)^{-\frac{1}{2}} F(-1/N \tau).
$$
As we will prove in the next section, $G(\tau)$ also has a Fourier expansion of the form \eqref{eq:fourier}, and we denote the coefficients of $G(\tau)$ using the same notation, with the subscript $G$. 

\begin{theorem}\label{thm:functional-equation}
Let $F \in V^{mg}_{\frac{1}{2}}(N,\chi )$ and let $G = F |_{\frac{1}{2}} W(N)$. Then $\Lambda_{N}(F,s)$ has a meromorphic continuation to the whole complex plane with simple poles at $s = -\frac{1}{2}$ and $1$, with residues $-N^{-\frac{1}{4}}d_{F}(2)$ and $ i^{\frac{1}{2}}  N^{-\frac{1}{4}}d_{G}(2)$ respectively. Additionally, $\Lambda_N(F,s)$ has a pole at $0$ and a pole at $\frac{1}{2}$ with residues $(-d_{F}(0) + d_{F}(1) \log {\sqrt{N}})$ and $i^{\frac{1}{2}} ( d_{G}(0) -d_{G}(1) \log {\sqrt{N}})$, respectively. The pole at $0$ is a double pole unless $d_F(1) = 0$, and the pole at $\frac{1}{2}$ is a double pole unless $d_G(1) = 0$. 
Moreover, we have the functional equation
\begin{align}\label{function-equation2}
 \Lambda_{N}(F,s)= i^{\frac{1}{2}} \Lambda_{N}\left(G,\frac{1}{2}-s\right).
\end{align}
\end{theorem}

We use these $L$-functions to prove a summation formula for coefficients of sesquiharmonic Maass forms, given below in Theorem \ref{thm:sumformula-sesqui}. Our method extends a recent preprint \cite{bdgrt}, who built on an older paper of Chandrasekharan and Narasimhan \cite{ChandNaras61}. In \cite{ChandNaras61}, a summation formula was obtained for a general class of Dirichlet series by applying Perron's formula and evaluating the contour integrals appearing explicitly in terms of Bessel functions. When applied to mock modular forms as in \cite{bdgrt}, the nonholomorphic part of the associated harmonic Maass forms contributes an additional term to the formula. In the case of sesquiharmonic Maass forms, we find that two additional terms appear in the formula, coming from the two non-holomorphic generating series in \eqref{eq:fourier}.  As a corollary of Theorem \ref{thm:sumformula-sesqui}, we obtain the following asymptotic formula for a certain weighted average of Hurwitz class numbers and the positive discriminant quadratic traces $\mathrm{Tr_d(1)}$ introduced by \cite{DukeImamogluToth}: 
\begin{theorem} \label{thm:asymptotic-formula-example}
For $\rho > \frac{3}{2}$ and any $\epsilon>0$, we have 
\begin{align}\label{eq:asymptotic-formula for h*(n)}
\sum_{n\leq x} \left(1-\frac{n}{x}\right)^{\rho} &\left( \mathrm{Tr_{n}(1)} + \frac{ H(n) ((x/n) -1)^{\frac{1}{2}}  }{\sqrt{2n} (\rho + \frac{1}{2})} { }_2F_1\left(1, \frac{1}{2}, \rho + \frac{3}{2}; \frac{1}{2} \left( 1 - \frac{x}{n} \right) \right)  \right) \nonumber \\
& = \frac{ \pi x }{6 \sqrt{2} (\rho+1)} + \mathcal{O}(x^{\frac{1}{2}+\epsilon}),
\end{align}
where $ { }_2F_1$ is the hypergeometric function given by the Gauss series:
$$
{ }_2F_1 ( a,b,c;z) := \frac{\Gamma(c)}{\Gamma(a) \Gamma(b)}  \sum_{s=0}^{\infty} \frac{\Gamma(a+s) \Gamma(b+s) }{\Gamma (c+s )} z^s .
$$
\end{theorem} 
 
\subsection{Outline of the paper} We begin with the background of modular forms and the theory of sesquiharmonic Maass forms in Section \ref{sec:background}. Here we give the Fourier expansion of sesquiharmonic Maass form $F$ of weight $\frac{1}{2}$ with polynomial growth and study the behavior of the nonconstant coefficients. Next in Section \ref{sec:L-functions}, we prove Theorem \ref{thm:functional-equation} and some growth bounds for the special functions appearing in \eqref{eq:complete-L-function}. Section \ref{sec:summation-formula} consists of the summation formula for weight $\frac{1}{2}$ sesquiharmonic Maass forms and, as a corollary, an asymptotic related to the summation formula. Then in Section \ref{sec:example}, we apply the summation formula and its corollary to prove Theorem \ref{thm:asymptotic-formula-example}.

 \section*{Acknowledgements}
The first author was partially supported by National Science Foundation Grant DMS-2401356 and Simons Foundation Grant \#953473.  The authors thank Nikolaos Diamantis for helpful comments on an earlier version of this manuscript.  
 
 \section{Modular Forms Background}\label{sec:background}
The Jacobi theta function is defined by
\begin{gather*}
  \theta(\tau) := \sum_{n=-\infty}^\infty q^{n^2}
\end{gather*}
where we use the notation 
\begin{gather*}
q:=e(\tau)=e^{2\pi i\tau},\  \ \ \tau \in \mathbb{H}, \ \ \ \Re(\tau) = u, \ \ \Im(\tau) =v.
\end{gather*}
The multiplier system for the theta function is given by
\begin{gather*}
  \nu_{\theta}\left( \begin{pmatrix} a & b \\ c & d \end{pmatrix} \right)
:=
  (c\tau + d)^{-\frac12} \, \frac{\theta(\gamma \tau)}{ \theta(\tau)}
=
  \left( \frac{c}{d} \right)  \epsilon_d^{-1},
\end{gather*}
where
\begin{gather*}
\epsilon_d = \begin{cases} 
1 & d\equiv 1 \pmod{4} , \\
i & d \equiv 3 \pmod{4}.
\end{cases}
\end{gather*}

For any $\gamma = \begin{bmatrix} a & b \\ c & d \end{bmatrix} \in \SL_2(\ZZ)$ we define the $|_k \gamma$ operator acting on functions $f: \mathbb{H} \to \mathbb{C}$ by
$$
f |_k \gamma (z) = (cz+ d)^{-k} f(\gamma z).
$$
For $k \in \mathbb{Z}$, this defines a group action of $\SL_2(\ZZ)$ on $\{f : \mathbb{H} \to \mathbb{C}\}$. 

We let $\Gamma_0(N)$ denote the usual congruence subgroup of $\SL_2(\ZZ)$, consisting of matrices in $\SL_2(\ZZ)$ whose lower left entry is divisible by $N$. Given a Dirichlet character $\chi$ modulo $N$ and any $\gamma \in \Gamma_0(N)$, $\chi(\gamma)$ denotes $\chi$ evaluated at the lower right entry of $\gamma$.

We have the {\it weight $k$ hyperbolic Laplace operator} is given by
    \begin{align*}
    \Delta_k := -v^2\left(\frac{\partial^2}{\partial u^2}+\frac{\partial^2}{\partial v^2}\right) + ikv\left(\frac{\partial}{\partial u} + i\frac{\partial}{\partial v}\right),
    \end{align*}
    and this operator decomposes as 
    \begin{align} \label{eq:Deltasplitting}
    \Delta_{k} = - \xi_{2-k} \xi_{k},
    \end{align}
    where $\xi_k$ acts on functions $f(\tau)$ by $$\xi_k(f)(\tau)=2iy^k\overline{\frac{\partial f}{\partial \overline\tau}}.$$
    In particular, holomorphic functions are automatically annihilated by $\Delta_k$.
    Further, this operator intertwines with the slash operator in ``dual weights'' $k$ and $2-k$ in the sense that for all $f\colon \mathbb H\rightarrow\mathbb C$ and all $\gamma\in\Gamma,$
    $$
    \xi_k(f|_k\gamma)=\left(\xi_k(f)\right)|_{2-k}\gamma   .
    $$ 
 
In the following definitions, given a cusp $\alpha = (a,b) \in \Gamma_0(N) \backslash \mathbb{P}^1 (\mathbb{Q})$, we let $\sigma_{\alpha} \in \SL_2(\mathbb{Z})$ be such that $\sigma_{\alpha} ( \infty ) = \alpha$. 

\begin{definition}
  Let $N \in \mathbb{Z}$ such $4 | N$ and $\chi$ be a character on $\Gamma_{0}(N)$. Also let $F$ be a real analytic function on $\mathbb{H}$. 
\begin{enumerate}[label=(\alph*)]
\item We say that $F$ is a holomorphic modular form of weight $k$ on $\Gamma_{0}(N)$ with character $\chi$ if it satisfies the following conditions:
	\begin{enumerate}[label=(\roman*)]
	\item  For all $\gamma \in \Gamma_{0}(N), F|_k \gamma = \chi(\gamma) F.$
	\item $\xi_{k}(F)=0$.
	\item  $F|_k \sigma_{\alpha} (\tau)$ has at most polynomial growth in $v$ as $v \to \infty$ for all $\alpha \in  \Gamma_0(N) \backslash \mathbb{P}^1 (\mathbb{Q})$ (i.e. $F$ has polynomial growth at all cusps).
	\end{enumerate}
\item We say that $F$ is a harmonic Maass form of weight $k$ on $\Gamma_{0}(N)$ with character $\chi$ if it satisfies the following conditions:
	\begin{enumerate}[label=(\roman*)]
	\item For all $\gamma \in \Gamma_{0}(N), 
	F|_k \gamma =
	\begin{cases} 
  	  \chi(\gamma)F, & \text{if } k \in \mathbb{Z} \\
    	 \chi(\gamma)\nu_{\theta}(\gamma)^{2k}F, & \text{if } k \in \frac{1}{2} + \mathbb{Z}.
	\end{cases}$
	\item $\Delta_{k}(F)=0$.
	\item  $F|_k \sigma_{\alpha} (\tau)$ has at most linear exponential growth as $v \to \infty$ for all $\alpha \in  \Gamma_0(N) \backslash \mathbb{P}^1 (\mathbb{Q})$.
	\end{enumerate}
\item We say that $F$ is a sesquiharmonic Maass form of weight $k$ with respect to $\Gamma_{0}(N)$ with character $\chi$ if it satisfies the following properties:
 	\begin{enumerate}[label=(\roman*)]
	  \item For all $\gamma \in \Gamma_{0}(N), 
	F|_k \gamma =
	\begin{cases} 
  	  \chi(\gamma)F, & \text{if } k \in \mathbb{Z} \\
    	 \chi(\gamma)\nu_{\theta}(\gamma)^{2k}F, & \text{if } k \in \frac{1}{2} + \mathbb{Z}.
	\end{cases}$
	\item $\Delta_{k}(F)$ lies in $M_{k}(N, \chi)$
	\item  $F|_k \sigma_{\alpha} (\tau)$ has at most linear exponential growth as $v \to \infty$ for all $\alpha \in  \Gamma_0(N) \backslash \mathbb{P}^1 (\mathbb{Q})$.
	\end{enumerate}
\end{enumerate}
\end{definition}
We denote the $\mathbb{C}$-vector spaces of functions satisfying the conditions in (a), (b), and (c) respectively by $M_k(N, \chi)$, $H_k(N,\chi)$, and $V_k(N,\chi)$, respectively. Next we define the subspace $V_k(N,\chi )$ consisting of sesquiharmonic Maass forms with moderate growth by
\begin{align}
	V^{mg}_k(N,\chi) := \left\{ F \in V_k(4N,\chi) : F \mbox{ has polynomial growth at all cusps} \right\}.
\end{align}

\subsection{Fourier Expansions with $k = \frac{1}{2}$}
To describe the Fourier expansion of these functions when $k=\frac{1}{2}$, we will use the following special functions defined in \cite{DukeImamogluToth}:

\begin{align}\label{expression-alpha}
\alpha{(y)} := \frac{\sqrt{y}}{4\pi} {\int_{0}^\infty t^{-1/2}e^{-{\pi} yt}\log(1+t)dt}  \hspace{0.5cm} y>0,
\end{align}
and 
\begin{align}\label{expression-beta}
 \beta{(y)} := \frac{1}{\sqrt{\pi}} \int_{{\pi}y}^\infty {t^{1/2 -1} e^{-t}} dt.
\end{align}

\begin{lemma}[Proposition 2.4 \cite{allen-beckwith-sharma}]
Let $k = \frac{1}{2}$ and $F \in V_{\frac{1}{2}}(4N,\chi)$. Then $F$ has a Fourier series expansion of the form 
\begin{align}
	F(\tau)= \left(d(0) +d(1)\log y+d(2)y^{\frac{1}{2}} \right)+ \sum_{\substack{n \gg - \infty, \\ n \neq 0} } {c(n) q^{n}}n\\
+ \sum_{\substack{n \ll \infty, \\ n \neq 0} } {b(n) \beta(-4ny)q^n} + \sum_{n \geq 1}{a(n)\alpha(4ny)q^n}.
 \end{align}

\end{lemma}
\begin{remark}
We note that Proposition 2.4 of \cite{allen-beckwith-sharma}, the formula also includes the term $d(3)y^{\frac{1}{2}} \log y$. However, a short calculation shows that in the sesquiharmonic space $ V_{\frac{1}{2}}(4N,\chi)$, this term must equal 0.
\end{remark}

\begin{lemma}\label{thm:fourier-moderate-growth}
Suppose $F \in V^{mg}_{\frac{1}{2}} (4 N, \chi)$ then $F$ has a Fourier expansion of the form
\begin{align}\label{fourier-expan}
	F(\tau)= \left(d(0) +d(1)\log v+d(2)v^{\frac{1}{2}} \right)+ \sum_{n> 0}  {c(n) q^{n}}
	\end{align}
	$$ + \sum_{n < 0}{b(n) \beta(-4nv)q^n} + \sum_{n \geq 1}{a(n)\alpha(4nv)q^n}$$.
\end{lemma}
\begin{proof}
If $c(n) \neq 0$ for some $n <0$, then $F(\tau)$ has exponential growth at $i \infty$, since $|q^n| = e^{-2 \pi n v}$. Similarly,  $|\beta(- 4 \pi n v) q^n| \sim v^{-\frac{1}{2}} e^{2 \pi n v}$ shows that $b(n) = 0$ for $n >0$. \end{proof}
We also need the growth conditions for the coefficients of sesquiharmonic maass forms with polynomial growth. 
 
\begin{lemma}\label{thm:polynomial-bound-coefficients}
If  $F \in V^{mg}_{\frac{1}{2}} (4 N,\chi)$ with the Fourier expansion as in \eqref{fourier-expan}, then $a(n) = \mathcal{O}(n^{{\mu}_{1}}), b(n) = \mathcal{O}(n^{{\mu}_{2}})$ and $c(n)= \mathcal{O}(n^{{\mu}_{3}})$ as $n \rightarrow \infty$ for some ${\mu}_{1},{\mu}_{2}, {\mu}_{3} \in \mathbb{R}_{>0}.$
\end{lemma}
\begin{proof}
We compute directly
\begin{align*}
&\xi_{\frac{1}{2}}(F)(\tau)= \\
&-(4{\pi})^{\frac{1}{2}} \sum_{n \geq 0}{\overline{b(n)}n^{\frac{1}{2}}q^n} + \frac{1}{8{\pi}}\sum_{n \geq 1}{a(n)\Gamma \left(-\frac{1}{2},4 {\pi}ny\right) (4{\pi}n)^{\frac{1}{2}} q^{-n}} +\overline{ d(1)} y^{-\frac{1}{2}} + \frac{\overline{d(2)}}{2} 
\end{align*} 
The basic properties of $\xi_{\frac{1}{2}}$ imply that $\xi_{\frac{1}{2}}(F)$ is harmonic, has weight $\frac{3}{2}$ with character $\chi$ for $\Gamma_0(N)$, and has moderate growth at all cusps. In the notation of \cite{bdgrt}, this means $\xi_{\frac{1}{2}}(F) \in \mathbb{H}^{Eis}_{\frac{3}{2}}{(N, \chi)}$. Now the bounds for $b(n)$ and $a(n)$ follow from Proposition 3.6 of \cite{bdgrt}. For $c(n)$, we notice that for $n > 0$,
\begin{align*}
c(n)e^{-2{\pi}ny}= \int_{0}^{1} {F(x+iy)e^{-2{\pi}nx}}dx
\end{align*}
Since $F$ has polynomial growth for all cusps, by Corollary 5.1.17 of \cite{cohen2017modular},  there exist $ A, C_{1}, C_{2}>0$ such that $|F(\tau)| \leq C_{1}$ as $v \rightarrow \infty $ and $|F(\tau)| \leq C_{2}v^{-A}$ as $ v \rightarrow 0$. Therefore, 
\begin{align*}
|c(n)e^{-2{\pi}nv}| \leq Cv^{-A}.
\end{align*}
Now by taking $v= 1/n$, we deduce $a(n)= \mathcal{O}(n^A)$.
\end{proof}

\subsection{Fricke Involution}
We have the analogue of Proposition 3.6 of \cite{bdgrt}:

\begin{lemma}
Let $F \in  V^{mg}_{\frac{1}{2}} ( N, \chi)$ and define $G(\tau) = F|_{\frac{1}{2}}W(N)(\tau)$. Then  $G \in   V^{mg}_{\frac{1}{2}} (N, \bar{\chi}\left(\frac{N}{\bullet})\right)$.
\end{lemma}
\begin{proof} 
The transformation law is a consequence of Proposition 1.4 of \cite{Shimura}. Moreover, from Proposition 3.4 of \cite{bdgrt}, we have 
\begin{align*}
\xi_{\frac{1}{2}}(F|_{\frac{1}{2}}W(N))= Ni \xi_{\frac{1}{2}}(F)|_{\frac{3}{2}}W(N).
\end{align*} 
Now after applying $\xi_{\frac{3}{2}}$, we get 
  
\begin{align*}
\xi_{\frac{3}{2}}(\xi_{\frac{1}{2}}(F|_{\frac{1}{2}}W(N)))&= \xi_{\frac{3}{2}} \left( Ni (\xi_{\frac{1}{2}}(F))|_{\frac{3}{2}}W(N)\right)\\
&= N^{2} \xi_{\frac{3}{2}}(\xi_{\frac{1}{2}}(F))|_{\frac{1}{2}}W(N)
\end{align*}

After again applying $\xi_{\frac{1}{2}}$, we obtain the required condition to be sesquiharmonic,
\begin{align*}
( \xi_{\frac{1}{2}}\circ \Delta_{\frac{1}{2}})F|_{\frac{1}{2}}W(N) =0.
\end{align*}
\end{proof}

\section{$L$-series attached to sesquiharmonic Maass forms}\label{sec:L-functions}
Throughout the section, we assume that $F$ and $G$ belong to $V^{mg}_{\frac{1}{2}}(N,\chi)$. By Lemma \ref{thm:fourier-moderate-growth}, they have Fourier expansions as in \eqref{eq:fourier}.  Also by Lemma \ref{thm:polynomial-bound-coefficients}, there is a real number $\alpha > 0$ such that $c_{F}(n),b_{F}(n),a_{F}(n),c_{G}(n),b_{G}(n),a_{G}(n)$ are all bounded by $\mathcal{O}(|n|^{\alpha}),n \in \mathbb{Z}$. 

\subsection{Special Functions}
To define our $L$-series, we will need a few special functions. 

First, for any $ \nu \in \mathbb{R}$, define 
\begin{equation}\label{eq:W-def}
W_{\nu}(s) = \displaystyle \int_{0}^ \infty \Gamma\left(\nu, 2x\right)e^{x} x^{s-1} dx, \hspace{0.5cm} \Re(s) >0,
\end{equation}

where $\Gamma(\nu, 2x)$ is the incomplete gamma function defined as follows
$$ \Gamma(s,a) = \int_{a}^\infty e^{-y} y^{s-1} dy, \hspace{1cm} a>0, \hspace{0.5cm} \Re(s)>0.$$
We also define the function
\begin{equation}\label{eq:Z-def}
 Z(s)= \int_{0}^{\infty} \alpha\left(\frac{2x}{\pi}\right)e^{-x}x^{s}\frac{dx}{x},   \hspace{0.5cm} \Re(s)>1.
\end{equation}

It will be helpful to relate the special function $W_{\nu}(s)$  to the $\beta$ special function introduced in the previous section and to study the analytic properties of the special function $Z(s)$.
\begin{lemma}\label{thm:W-formula}
For $\Re(s) >0$,
$$
W_{\frac{1}{2}}(s) = \sqrt{\pi} \int_0^{\infty} \beta \left( \frac{2 x}{\pi} \right) e^{x} x^{s-1} dx 
$$
\end{lemma}
\begin{proof}
Using \eqref{expression-beta}, we have 
\begin{align*}
\int_{0}^{\infty} \beta\left(\frac{2x}{\pi}\right)e^{x}x^{s}\frac{dx}{x}
 &= \int_{0}^{\infty} \frac{1}{\sqrt{\pi}} \Gamma\left(\frac{1}{2} , 2x\right) e^{x}x^{s}\frac{dx}{x}\\
 &= \frac{1}{\sqrt{\pi}} W_{\frac{1}{2}}(s)
\end{align*}
\end{proof}

\begin{lemma}\label{thm:Z-formula}
$ For \Re(s) >1,$ we have 
\begin{align*}
Z(s)= \frac{\Gamma \left(s+\frac{1}{2} \right)}{\sqrt{8 {\pi}^3}} \int_{0}^{\infty} \frac{t^{-\frac{1}{2}} \log (1+t)}{(2t+1)^{s+\frac{1}{2}}} dt.
\end{align*}
\end{lemma}
\begin{proof}
Beginning with the definition of $Z(s)$, we compute
\begin{align*}
Z(s) &=  \int_{0}^{\infty} \alpha\left(\frac{2x}{\pi}\right)e^{-x}x^{s}\frac{dx}{x} \\
 &= \frac{1}{\sqrt{8 {\pi}^{3}}}\int_{0}^{\infty} \left(\sqrt{x}\left(\int_{0}^{\infty} t^{-\frac{1}{2}} e^{-2xt} \log (1+t) dt \right)\right) e^{-x}x^{s}\frac{dx}{x}.
\end{align*}
Now applying Fubini's theorem, we get 
\begin{align*}
Z(s)&=  \frac{1}{\sqrt{8 {\pi}^3}} \int_{0}^{\infty} t^{-\frac{1}{2}} \log (1+t)\left(\int_{0}^{\infty} \sqrt{x} e^{-2xt} e^{-x}x^{s}\frac{dx}{x}\right)dt\\
 &=  \frac{1}{\sqrt{8 {\pi}^3}} \int_{0}^{\infty} {t^{-1/2} \log (1+t) }\frac{\Gamma \left(s+\frac{1}{2}\right)}{(2t+1)^{s+\frac{1}{2}}} dt \\
 \end{align*} 
where we've used the change of variables $x(2t+1)\rightarrow y$.

\end{proof}

 \begin{lemma}\label{thm:holo-Z(s)}
$Z(s)$ is holomorphic on $\Re(s) >1$ and it has an analytic continuation to $\Re(s) >0.$
\end{lemma}
\begin{proof}
  Let $f(x) = \alpha\left(\frac{2x}{\pi}\right) e^{-x}$. Using the Lebesgue dominated convergence theorem, we find that $f(x)$ is continuous on $(0, \infty)$. 
  
Now note that as $ x \rightarrow 0$,
\begin{align*}
     &\int_{0}^{\infty} {t^{-1/2} e^{-2xt}\log (1+t)}dt 
     = x^{-\frac{1}{2}} \int_{0}^{\infty} u^{-\frac{1}{2}} e^{-2u} \log (1+\frac{u}{x}) du = \mathcal{O}(x^{-\frac{3}{2}}).
\end{align*}

It follows that $f(x) = O(x^{-1})$ as $x \rightarrow 0$.
On the other hand, $f(x)$ tends to $0$ as $x \to \infty$.  Therefore  by using Proposition 3.1.22 of \cite{cohen2017modular},
we deduce that $ Z(s)$ converges absolutely $\Re(s) > 1$ and is holomorphic in this region. \\

For $\Re(s) >0$, we define $Z(s)$ by the formula in Lemma \ref{thm:Z-formula}, 
\begin{align*}
    Z(s):= \frac{\Gamma(s+\frac{1}{2})}{\sqrt{8 \pi^3}} \int_{0}^{\infty} \frac{t^{-1/2} \log(1+t)}{(2t+1)^{s+\frac{1}{2}}}dt.
\end{align*}
 To see that the function on the right hand side is holomorphic, note that 
  \begin{align*}
      Z(s)&=  \frac{\Gamma(s+\frac{1}{2})}{\sqrt{8 \pi^3}} \sum_{N \geq 0} \int_{N}^{N+1} \frac{t^{-1/2} \log(1+t)}{(2t+1)^{s+\frac{1}{2}}}dt
  \end{align*}
Since $F_{N}(s)=  \int_{N}^{N+1} \frac{t^{-1/2}\log(1+t)}{(2t+1)^{s+\frac{1}{2}}}dt$ is holomorphic for each $N$ on $\Re(s) >0$, and the series $f(s):=\sum_{N \geq 0}F_{N}(s)$ converges uniformly on any compact subset in the region $\{s \in \mathbb{C} :\Re(s) >0\}$, $f(s)$ is holomorphic on $\Re(s)>0.$  This completes the proof.
\end{proof}

\begin{lemma}\label{thm:Z-formula-2}
For $\Re(s) > 0$,
$$
Z(s) = \frac{ \Gamma ( s + \frac{1}{2} )}{4 \pi^{\frac{3}{2}}} \int_0^{\infty} \mathbf{B}_{\frac{1}{2x+1}} \left(s, \frac{1}{2} \right) \frac{dx}{1+x}
$$
where $\mathbf{B}_{z}(a,b) = \int_{0}^{z}t^{a-1}(1-t)^{b-1}dt $ is the incomplete beta function.
\end{lemma}
\begin{proof}
 Note that 
\begin{align*}
\log (1+t) = \int_{0}^{t} \frac{dx}{1+x}
\end{align*}
Then as an application of Fubini's theorem, we can write $Z(s)$ as follows
\begin{align*}
Z(s)&=  \frac{\Gamma \left(s+\frac{1}{2}\right)}{\sqrt{8 {\pi}^3}} \int_{0}^{\infty} \frac{t^{-1/2}}{(2t+1)^{s+\frac{1}{2}}} \left(\int_{0}^{t} \frac{dx}{1+x}\right) dt\\
&= \frac{\Gamma \left(s+\frac{1}{2}\right)}{\sqrt{8 {\pi}^3}} \int_{0}^{\infty} \left(\int_{x}^{\infty}  \frac{t^{-1/2}}{(2t+1)^{s+\frac{1}{2}}} dt \right)  \frac{dx}{1+x}\\
\end{align*}
A change of variables allows us to rewrite the inner integral in terms of the incomplete beta function, proving the result.
 \end{proof}

\begin{proposition}\label{thm:finite-sum-formula}

Let $\alpha, \rho, r$ be positive real numbers with $\rho > \frac{1}{2}$.  Assume that on the line $R(s) = \alpha, L(s) = \sum_{n=1}^{\infty} c(n) n^{-s}$ is absolutely convergent. Then we have
\begin{align}
 \sum_{n \leq r} c(n) h_{\rho}(n,r) =  \int_{(\alpha)} \frac{Z(s) L(s)}{\Gamma(\rho+1+s)} r^{s}  ds 
\end{align}
where 
\begin{align*}
h_{\rho}(n,r) = \frac{1}{4 \pi^{\frac{3}{2}}} \frac{2 \pi i}{\Gamma(\rho + \frac{1}{2})} \left(\frac{r-n}{r} \right)^{\rho} \int_{0}^{1} \log \left(1+\frac{(t-1)(n-r)}{2n}\right)   t^{\rho - \frac{1}{2}} (1-t)^{-\frac{1}{2}} dt. 
\end{align*}

\end{proposition}
\begin{proof}
Here using Lemma \ref{thm:Z-bound} and Stirling's formula \protect{\cite[(5.11.9)]{NIST:DLMF}}, we deduce that the integral $\int_{(\alpha)} \frac{Z(s)}{\Gamma(\rho+1+s)} r^{s}  ds$ is absolutely convergent for $ \rho > - \frac{1}{2}$ and with this and the absolute convergence of $L(s)$, we can interchange the summation and integration as follows:

\begin{align}\label{int-sum of Z(s)}
 \int_{(\alpha)}\frac{Z(s)L(s)}{\Gamma(\rho+1+s)} r^s  ds = \sum_{n=1}^{\infty} c(n) \int_{(\alpha)} \frac{Z(s)}{\Gamma(\rho+1+s)} \left(\frac{r}{n}\right)^{s}  ds 
\end{align}
Let $y = \frac{r}{n}$ with $y>1$. Then using Lemma \ref{thm:Z-formula-2}, we have
\begin{align*}
 \int_{(\alpha)} \frac{Z(s)}{\Gamma(\rho+1+s)} y^{s}  ds = \int_{(\alpha)} \left( \frac{ \Gamma ( s + \frac{1}{2} )}{4 \pi^{\frac{3}{2}}} \int_0^{\infty} \mathbf{B}_{\frac{1}{2u+1}} \left(s, \frac{1}{2} \right) \frac{du}{1+u} \right) \frac{y^{s} ds}{\Gamma(1+\rho+s)} 
 \end{align*}
We will use Fubini's theorem to swap the order of summation. To justify this, first note that we have the bound
$$
| \mathbf{B}_{\frac{1}{2u+1}} \left(s, \frac{1}{2} \right) | \le \mathbf{B}_{\frac{1}{2u+1}} \left(\alpha, \frac{1}{2} \right) < \infty
$$
for $\alpha > 0$.
So we compute
\begin{align*}
& \int_{(\alpha)} \int_{0}^{\infty} | \frac{ \Gamma ( s + \frac{1}{2} )}{\Gamma(1+\rho+s)} | \cdot  |\mathbf{B}_{\frac{1}{2u+1}} \left(s, \frac{1}{2} \right) \frac{y^{s}}{1+u}|  du ds \\
 & \le \int_{-\infty}^{\infty} \int_{0}^{\infty} | \frac{ \Gamma ( \alpha + iT + \frac{1}{2} )}{\Gamma(1+\rho+\alpha+iT)} | \cdot  \mathbf{B}_{\frac{1}{2u+1}} \left( \alpha , \frac{1}{2} \right) \frac{y^{\alpha}}{1+u}  du dT \\
  & = y^{\alpha} \int_{-\infty}^{\infty}  | \frac{ \Gamma ( s + \frac{1}{2} )}{\Gamma(1+\rho+s)} | dT \cdot \int_{0}^{\infty} \cdot  \mathbf{B}_{\frac{1}{2u+1}} \left( \alpha , \frac{1}{2} \right) \frac{du}{1+u} 
\end{align*}
Stirling's formula \protect{\cite[(5.11.9)]{NIST:DLMF}} ensures the convergence of the first integral for $\rho > \frac{1}{2}$, and the convergence of the second integral follows from the estimate $\mathbf{B}_{\frac{1}{2u+1}} \left( \alpha , \frac{1}{2} \right) = O(u^{-1})$ as $u \to \infty$. 
Now applying Fubini's Theorem, we obtain
 \begin{align*}
 \int_{(\alpha)} \frac{Z(s)}{\Gamma(\rho+1+s)} y^{s}  ds &= \frac{1}{4 \pi^{\frac{3}{2}}} \int_{(\alpha)} \int_{0}^{\infty} \frac{ \Gamma ( s + \frac{1}{2} )}{\Gamma(1+\rho+s)}\mathbf{B}_{\frac{1}{2u+1}} \left(s, \frac{1}{2} \right) \frac{y^{s}}{1+u} du ds\\
&=  \frac{1}{4 \pi^{\frac{3}{2}}}\int_{0}^{\infty}   \int_{0}^{\frac{1}{1+2u}} t^{-1} (1-t)^{-\frac{1}{2}} \left( \int_{(\alpha)}  \frac{ \Gamma ( s + \frac{1}{2} )}{\Gamma(1+\rho+s)} (ty)^{s} ds \right) dt \frac{du}{1+u}
\end{align*}
Now it is known from Section 7.3, eq. 20 of \cite{bateman1954tables} that for $\gamma , a >0$, 
\begin{align*}
\frac{1}{2 \pi i} \int_{(\gamma)} \frac{\Gamma(s)}{\Gamma(s+a)} y^{-s} ds = \frac{(1-y)^{a-1}}{\Gamma(a)} \mathbf{1}_{(0,1)} (y) 
\end{align*}
where $\mathbf{1}_{A}$ is the indicator function for a set A. \\
Therefore, 
\begin{align*}
 \int_{(\alpha)}  \frac{ \Gamma ( s + \frac{1}{2} )}{\Gamma(1+\rho+s)} (ty)^{s} ds = \int_{(\alpha +\frac{1}{2})}  \frac{ \Gamma (v)}{\Gamma( v+\rho+\frac{1}{2})} \left(\frac{1}{ty}\right)^ {- (v-\frac{1}{2})}dv  \\
=   \left(\frac{1}{ty}\right) ^{\frac{1}{2}}  \frac{2 \pi i}{\Gamma(\rho + \frac{1}{2})}  \left(1- \frac{1}{ty} \right)^{\rho +\frac{1}{2}-1} \mathbf{1}_{(0,1)} \left(\frac{1}{ty}\right)  
\end{align*}
for $\alpha > -\frac{1}{2}$ and $\rho > -\frac{1}{2}$.\\
Then, 
\begin{align*}
& \int_{(\alpha)} \frac{Z(s)}{\Gamma(\rho+1+s)} y^{s}  ds = \\ 
&\frac{1}{4 \pi^{\frac{3}{2}}}\int_{0}^{\infty} \left(  \frac{2 \pi i}{\Gamma(\rho + \frac{1}{2})} \int_{0}^{\frac{1}{1+2u}} t^{-1} (1-t)^{-\frac{1}{2}}  \left(\frac{1}{ty}\right) ^{\frac{1}{2}} \left(1- \frac{1}{ty} \right)^{\rho-\frac{1}{2}}\mathbf{1}_{(0,1)} \left(\frac{1}{ty}\right)  dt \right) \frac{du}{1+u} \\
&=\frac{1}{4 \pi^{\frac{3}{2}}}\int_{0}^{\infty} \frac{2 \pi i}{\Gamma(\rho + \frac{1}{2})} \left(\int_{\frac{1}{y}}^{\frac{1}{1+2u}} t^{-1} (1-t)^{-\frac{1}{2}}  \left(\frac{1}{ty}\right) ^{\frac{1}{2}} \left(1- \frac{1}{ty} \right)^{\rho-\frac{1}{2}} dt  \right)\frac{du}{1+u}\\
\end{align*}
Now by using the change of variables $u = \frac{y}{1-y} \cdot \left( 1 - \frac{1}{ty} \right)$, we compute
\begin{align*}
\int_{\frac{1}{y}}^{\frac{1}{1+2u}} t^{-1} (1-t)^{-\frac{1}{2}}  \left(\frac{1}{ty}\right) ^{\frac{1}{2}} \left(1- \frac{1}{ty} \right)^{\rho-\frac{1}{2}} dt=  \left(\frac{y-1}{y} \right)^{\rho}  \mathbf{B}_{\frac{y-1-2u}{y-1}} \left(\rho+\frac{1}{2}, \frac{1}{2} \right)
\end{align*}
Therefore,  
\begin{align*}
 \int_{(\alpha)} \frac{Z(s)}{\Gamma(\rho+1+s)} y^{s}  ds = \frac{1}{4 \pi^{\frac{3}{2}}} \frac{2 \pi i}{\Gamma(\rho + \frac{1}{2})} \int_{0}^{\infty}  \left(\frac{y-1}{y} \right)^{\rho}  \mathbf{B}_{\frac{y-1-2u}{y-1}} \left(\rho+\frac{1}{2}, \frac{1}{2} \right) \frac{du}{1+u} \\
= \frac{1}{4 \pi^{\frac{3}{2}}} \frac{2 \pi i}{\Gamma(\rho + \frac{1}{2})} \left(\frac{y-1}{y} \right)^{\rho}  \int_{0}^{\infty} \mathbf{B}_{\frac{y-1-2u}{y-1}} \left(\rho+\frac{1}{2}, \frac{1}{2} \right) \frac{du}{1+u}\\
=  \frac{1}{4 \pi^{\frac{3}{2}}} \frac{2 \pi i}{\Gamma(\rho + \frac{1}{2})} \left(\frac{y-1}{y} \right)^{\rho} \int_{0}^{\infty} \int_{0}^{1- \frac{2u}{y-1}} t^{\rho - \frac{1}{2}} (1-t)^{-\frac{1}{2}} \frac{dt du}{1+u}
\end{align*}
Here since the right hand side is absolutely convergent for $ \rho > -\frac{1}{2}$, we can swap the two integrals and get the following,
\begin{align*}
 \int_{(\alpha)} \frac{Z(s)}{\Gamma(\rho+1+s)} y^{s} ds =  \frac{1}{4 \pi^{\frac{3}{2}}} \frac{2 \pi i}{\Gamma(\rho + \frac{1}{2})} \left(\frac{y-1}{y} \right)^{\rho} \int_{0}^{1} \left(\int_{0}^{\frac{(t-1)(1-y)}{2}} \frac{du}{1+u} \right) t^{\rho - \frac{1}{2}} (1-t)^{-\frac{1}{2}} dt\\
= \frac{1}{4 \pi^{\frac{3}{2}}} \frac{2 \pi i}{\Gamma(\rho + \frac{1}{2})} \left(\frac{y-1}{y} \right)^{\rho} \int_{0}^{1} \log \left(1+\frac{(t-1)(1-y)}{2}\right)   t^{\rho - \frac{1}{2}} (1-t)^{-\frac{1}{2}} dt 
\end{align*}
Now substituting $y = \frac{r}{n}$, we have
\begin{align*}
 \int_{(\alpha)} \frac{Z(s)}{\Gamma(\rho+1+s)} \left(\frac{r}{n}\right)^{s} ds =\frac{1}{4 \pi^{\frac{3}{2}}} \frac{2 \pi i}{\Gamma(\rho + \frac{1}{2})} \left(\frac{r-n}{r} \right)^{\rho} \int_{0}^{1} \log \left(1+\frac{(t-1)(n-r)}{2n}\right)   t^{\rho - \frac{1}{2}} (1-t)^{-\frac{1}{2}} dt 
\end{align*}
as required. 
\end{proof}

\begin{lemma}\label{thm:finite-sum-bound}
Let $\alpha > 0 , \rho > \frac{1}{2}$ be such that $L(s) = \sum_{n\geq1} a(n) n^{-s}$ is convergent absolutely on $Re(s) = \alpha$. Then we have 
\begin{align}
 \sum_{n \leq x} a(n) h_{\rho}(n,x) = \mathcal{O}(x^\alpha)  \hspace{1cm}  \ x \rightarrow \infty. 
\end{align}
\end{lemma}
\begin{proof}
From Proposition \ref{thm:finite-sum-formula}, we have
\begin{align*}
\big|\sum_{n \leq x} a(n) h_{\rho}(n,x)\big| &=  \left|\int_{(\alpha)} \frac{Z(s) L(s)}{\Gamma(\rho+1+s)} x^{s}  ds\right| \\
& \leq x^{\alpha} \left|\int_{(\alpha)} \frac{Z(s) L(s)}{\Gamma(\rho+1+s)} ds\right|.
\end{align*}
 Now since $|L(s)| = \mathcal{O}(1)$ and the integral $\int_{(\alpha)} \frac{Z(s)}{\Gamma(\rho+1+s)} ds$ is absolutely convergent for $\rho >- \frac{1}{2}$, we get the required bound.
\end{proof}

\subsection{Proof of Theorem \ref{thm:functional-equation}}

The first step is to compute the Mellin transforms of all the nonconstant Fourier coefficients of $F$. 
First, we have 
 \begin{align}
 c_{F}(n) \int_{0}^{\infty} e^{\frac{- 2 \pi nt}{\sqrt{N}}}t^{s}\frac{dt}{t} = \left(\frac{2 \pi}{\sqrt{N}}\right)^{-s} \Gamma(s) \frac{c_{F}(n)}{n^{s}}.
\end{align}
and
\begin{align*}
b_{F}(-n) \int_{0}^{\infty} \beta\left(\frac{4nt}{\sqrt{N}}\right)e^{\frac{ 2 \pi nt}{\sqrt{N}}}t^{s}\frac{dt}{t}= \left(\frac{2 \pi n}{\sqrt{N}}\right)^{-s} b_{F}(-n) \left(\int_{0}^{\infty} \beta\left(\frac{2x}{\pi}\right)e^{x}x^{s}\frac{dx}{x}\right)
\end{align*}
where we use the change of variable $ \frac{2 \pi nt}{\sqrt{N}} \rightarrow x .$ Applying Lemma \ref{thm:W-formula}, we obtain
\begin{align}
b_{F}(-n) \int_{0}^{\infty} \beta\left(\frac{4nt}{\sqrt{N}}\right)e^{\frac{ 2 \pi nt}{\sqrt{N}}}t^{s}\frac{dt}{t}= \left(\frac{2 \pi n}{\sqrt{N}}\right)^{-s} b_{F}(-n)\frac{1}{\sqrt{\pi}} W_{\frac{1}{2}}(s).
\end{align}

Similarly, using the change of variables $ \frac{2 \pi nt}{\sqrt{N}} \rightarrow x$, we get
\begin{align*}
a_{F}(n) \int_{0}^{\infty} \alpha\left(\frac{4nt}{\sqrt{N}}\right)e^{\frac{- 2 \pi nt}{\sqrt{N}}}t^{s}\frac{dt}{t}= \left(\frac{2 \pi n}{\sqrt{N}}\right)^{-s} a_{F}(n)\left( \int_{0}^{\infty} \alpha\left(\frac{2x}{\pi}\right)e^{-x}x^{s}\frac{dx}{x}\right)
\end{align*}

 Therefore, using Lemma \ref{thm:Z-formula}, 
 \begin{align*}
 \left(\frac{2 \pi n}{\sqrt{N}}\right)^{-s} a_{F}(n) Z(s) &=  \left(\frac{2 \pi n}{\sqrt{N}}\right)^{-s} a_{F}(n)  \int_{0}^{\infty} \alpha\left(\frac{2x}{\pi}\right)e^{-x}x^{s}\frac{dx}{x } \\
 &= a_{F}(n) \int_{0}^{\infty} \alpha\left(\frac{4nt}{\sqrt{N}}\right)e^{\frac{- 2 \pi nt}{\sqrt{N}}}t^{s}\frac{dt}{t}. 
 \end{align*}

Now for $Re(s)> \alpha +1$ we have 
\begin{align*} 
& \int_{0}^{\infty} \Bigg( F\left(\frac{it}{\sqrt{N}}\right) - d_{F}(0)- d_{F}(1) \log\left(\frac{t}{\sqrt{N}}\right)- d_{F}(2)\left(\frac{t}{\sqrt{N}}\right)^{\frac{1}{2}}  \Bigg) t^{s} \frac{dt}{t}  \\ = 
& \int_{0}^{\infty} \left( \sum_{n=1}^{\infty} c_{F}(n) e^{\frac{- 2 \pi nt}{\sqrt{N}}}\right)t^{s}\frac{dt}{t} + \int_{0}^{\infty} \left( \sum_{n=1}^{\infty} b_{F}(-n) \beta(4nt/{\sqrt{N}}) e^{\frac{2 \pi nt}{\sqrt{N}}}\right)t^{s}\frac{dt}{t} +\\ & \int_{0}^{\infty} \left( \sum_{n=1}^{\infty} a_{F}(n) \alpha{(4nt/{\sqrt{N}})} e^{\frac{- 2 \pi nt}{\sqrt{N}}}\right)t^{s}\frac{dt}{t} 
\\ = & \left(\frac{\sqrt{N}}{2 \pi}\right)^{s} \left[\Gamma(s)L^{+}(F,s) + \frac{1}{\sqrt{\pi}} W_{\frac{1}{2}}(s)L^{-}(F,s)+ Z(s)L^{*}(F,s)\right] = \Lambda_{N}(F,s).
\end{align*} 
Now, let $F^*(\tau)= F(\tau) - d_{F}(0) 
- d_{F}(1) \log(y) - d_{F}(2)y^{\frac{1}{2}}$ and $G^*(\tau) = G(\tau) - d_{G}(0) 
- d_{G}(1) \log(y)-d_{G}(2)y^{\frac{1}{2}}.$
 \\
Now for $Re(s) > \alpha +1$, we find that
 
\begin{align}
\Lambda_{N}(F,s)= \int_{0}^{\infty}F^*\left(\frac{it}{\sqrt{N}}\right)t^{s} \frac{dt}{t}
= \int_{0}^{1} + \int_{1}^{\infty}F^*\left(\frac{it}{\sqrt{N}}\right)t^{s} \frac{dt}{t}.
\end{align}
Now,
\begin{align*}
\int_{0}^{1} F^*\left(\frac{it}{\sqrt{N}}\right)t^{s} \frac{dt}{t} &= \int_{1}^{\infty}F\left(\frac{i}{\sqrt{N}t}\right)t^{-s} \frac{dt}{t} -  \frac{d_{F}(0)}{s}+ d_{F}(1) \left(\frac{1}{s^{2}} +\frac{log{\sqrt{N}}}{s}\right)-  \notag \\ & \quad \frac{d_{F}(2)}{N^{1/4} (s+\frac{1}{2})}.
\end{align*}
On the other hand, since $F\left(\frac{i}{\sqrt{N}t}\right)=(it)^\frac{1}{2}  G\left(\frac{it}{\sqrt{N}}\right),$ we get 
\begin{align*}
\int_{1}^{\infty}F\left(\frac{i}{\sqrt{N}t}\right)t^{-s-1} dt=  i^\frac{1}{2} \int_{1}^{\infty} G\left(\frac{it}{\sqrt{N}}\right)t^{-s-\frac{1}{2}}dt.
\end{align*}
Therefore, 
 \begin{align}\label{eq:Lambda-formula}
\Lambda_{N}(F,s) &= \int_{1}^{\infty}F^*\left(\frac{it}{\sqrt{N}}\right)t^{s} \frac{dt}{t} + i^{\frac{1}{2}} \int_{1}^{\infty} G\left(\frac{it}{\sqrt{N}}\right)t^{-s-\frac{1}{2}}dt  - \frac{d_{F}(0)}{s}  \notag \\ 
& \quad  + d_{F}(1) \left(\frac{1}{s^{2}} +\frac{log{\sqrt{N}}}{s}\right)- \frac{d_{F}(2)}{N^{1/4} (s+\frac{1}{2})}  \notag \\
&=  \int_{1}^{\infty} F^*\left(\frac{it}{\sqrt{N}}\right) t^{s} \frac{dt}{t} + i^\frac{1}{2} \int_{1}^{\infty}  G^*\left(\frac{it}{\sqrt{N}}\right) t^{-s} \frac{dt}{t^{\frac{1}{2}}} + i^{\frac{1}{2}} \frac{d_{G}(0)}{(s-\frac{1}{2})} \notag \\ 
&\quad +  i^{\frac{1}{2}} d_{G}(1) \left(\frac{1}{(s-\frac{1}{2})^{2}} - \frac{\log(\sqrt{N})}{(s-\frac{1}{2})}\right) + i^{\frac{1}{2}} \frac{d_{G}(2)}{N^{\frac{1}{4}}} \frac{1}{(s-1)}  \notag \\ &\quad - \frac{d_{F}(0)}{s}+ d_{F}(1) \left(\frac{1}{s^{2}} + \frac{\log(\sqrt{N})}{s}\right)- \frac{d_{F}(2)}{N^{\frac{1}{4}}} \frac{1}{(s+\frac{1}{2})}.  \notag \\
\end{align}
Because of exponential decay as $t \rightarrow \infty$ of $F^*(it)$ and $G^*(it)$, the above two integrals are entire. Therefore, the poles and residues of $\Lambda_{N}(F,s)$ are obtained from the remaining terms.


 
To prove the functional equation, we replace $F$ with $G$ in the above expression and use $G|_{\frac{1}{2}} W(N)(\tau) = i^{-1}F(\tau)$. Thus we obtain
\begin{align*}
\Lambda_{N}&(G,s) = \int_{1}^{\infty} G^*\left(\frac{it}{\sqrt{N}}\right) t^{s} \frac{dt}{t} + i^{-\frac{1}{2}} \int_{1}^{\infty}  F^*\left(\frac{it}{\sqrt{N}}\right) t^{-s} \frac{dt}{t^{\frac{1}{2}}} - \frac{d_{G}(0)}{s}\notag \\ 
&\quad + d_{G}(1) \left(\frac{1}{s^{2}} + \frac{\log(\sqrt{N})}{s}\right)- \frac{d_{G}(2)}{N^{\frac{1}{4}}} \frac{1}{s+\frac{1}{2}} - i^{-\frac{1}{2}} \frac{d_{F}(0)}{(\frac{1}{2}-s)} \notag \\ 
&\quad +  i^{-\frac{1}{2}} d_{F}(1) \left(\frac{1}{(\frac{1}{2}-s)^{2}} + \frac{\log(\sqrt{N})}{(\frac{1}{2}-s)}\right)- i^{-\frac{1}{2}} \frac{d_{F}(2)}{N^{\frac{1}{4}}} \frac{1}{(1-s)} .
\end{align*} 
Now by replacing $s$ by $\frac{1}{2}-s$, we obtain the functional equation,
 $$ \Lambda_{N}(F,s)= i^{\frac{1}{2}} \Lambda_{N}(G,\frac{1}{2}-s).$$
This completes the proof.


\subsection{Growth rate}
  The next three bounds will be crucial for establishing the summation formula in Theorem \ref{thm:sumformula-sesqui}.
 \begin{lemma}\label{thm:Z-bound}
 For $\alpha > 0$, we have
 \begin{align*}
 Z(\alpha + iT) \ll_{\alpha}  |T|^{\alpha -1} e^{-\frac{{\pi}|T|}{2}} 
 \end{align*}
  as $|T| \to \infty$.
  \end{lemma}
\begin{proof}
From Lemma \ref{thm:Z-formula-2},  for $\Re(s) >0$, we have 
\begin{align*}
Z(s)= \frac{\Gamma \left(s+\frac{1}{2}\right)}{4  {\pi}^\frac{3}{2}} \int_{0}^{\infty}   \mathbf{B}_{\frac{1}{2x+1}}\left(s,\frac{1}{2}\right)   \frac{dx}{1+x}.
\end{align*}
Now from  8.17.7 and 15.2(i) of \cite{nist}, we have the asymptotic expression 
\begin{align*}
\mathbf{B}_{z}(a,b) \sim \frac{z^{a}}{a}  \hspace{0.5cm}\text{as} \ z \rightarrow 0^{+}.
\end{align*} 
Using this, we have 
\begin{align*}
 \mathbf{B}_{\frac{1}{2x+1}}\left(s,\frac{1}{2}\right) \sim \frac{1}{s}\left({\frac{1}{2x+1}}\right)^{s} \hspace{0.5cm}  x \rightarrow \infty. 
\end{align*}

Now
\begin{align*}
|Z(s)| &=  \bigg|\frac{\Gamma \left(s+\frac{1}{2} \right)}{4  {\pi}^\frac{3}{2}} \int_{0}^{\infty}   \mathbf{B}_{\frac{1}{2x+1}}\left(s,\frac{1}{2}\right)   \frac{dx}{1+x}\bigg|\\
& \leq \bigg|\frac{\Gamma \left(s+\frac{1}{2} \right)}{4  {\pi}^\frac{3}{2}} \bigg| \int_{0}^{\infty}  \bigg| \mathbf{B}_{\frac{1}{2x+1}}\left(s,\frac{1}{2}\right)   \frac{1}{1+x}\bigg| dx \\
& \leq \bigg|\frac{\Gamma \left(s+\frac{1}{2} \right)}{4  {\pi}^\frac{3}{2}} \bigg|  \int_{0}^{\infty}  \bigg|\frac{1}{s}\left({\frac{1}{2x+1}}\right)^{s} \frac{1}{1+x} \bigg| dx. 
\end{align*}
 For $s = \alpha + iT$, as $T \to \infty$, we obtain 
\begin{align*}
|Z(\alpha + iT)|  &\leq \bigg|\frac{\Gamma \left(\alpha+\frac{1}{2} + iT \right)}{4  {\pi}^\frac{3}{2}} \bigg|   \frac{1}{|T|} \int_{0}^{\infty}\left({\frac{1}{2x+1}}\right)^{\alpha} \frac{dx}{1+x} \\
& \leq \bigg|\frac{\Gamma \left(\alpha+\frac{1}{2} + iT \right)}{4  {\pi}^\frac{3}{2}} \bigg|   \frac{1}{|T|} \left(\int_{0}^{\frac{1}{2}} u^{\alpha -1} (1-u)^{-\alpha} du \right) \ll |T|^{\alpha -1} e^{-\frac{{\pi}|T|}{2}}.
\end{align*}
Here the last line follows from Stirling's formula.
\end{proof}

 \begin{lemma}\label{thm:LambdaStirlingBound}
    Let $F \in V^{mg}_{\frac{1}{2}}(4N,\chi)$, and let $\alpha > 1 + \operatorname{max} \{\mu_F^{1},\mu_F^{2},\mu_F^{3} \}$. As $|T| \to \infty$,
    $$
    \Lambda_{N}(F, \alpha+iT)=O_{\alpha,F} (|T|^{\alpha} e^{- \pi |T| /2}).
    $$
    \end{lemma}
    \begin{proof} Since $\alpha > 1 +  \operatorname{max} \{\mu_F^{1},\mu_F^{2},\mu_F^{3} \}$,
    $|L^{+} (F, \alpha + iT)| \le \sum_{n = 1}^{\infty} |c_{F}(n)|/n^{\alpha} =O(1).$ Similarly, $|L^{-} (F, \alpha + iT)| = O(1)$ and $|L^{*} (F, \alpha + iT)|= O(1).$ With Stirling's formula, Lemma \ref{thm:Z-bound} and Corollary 2.6 in \cite{bdgrt}, we have as $|T| \to \infty$: 
\begin{align*}
& \Lambda_{N}(F, \alpha+iT) = \\
 & \left( \frac{\sqrt{N}}{2 \pi } \right)^{\alpha + iT}  \left( \Gamma(\alpha + iT) L^+(F, \alpha + iT) +\frac{1}{\sqrt{\pi}} W_{\frac{1}{2}} (\alpha + iT) L^-(F,\alpha+iT)  +Z(\alpha + iT) L^*(F,\alpha+iT) \right)\\
    & \ll_{F, \alpha} |T|^{\alpha-1/2}e^{-\pi|T|/2} + |T|^{\alpha} e^{-\pi |T| /2}  +  |T|^{\alpha -1} e^{-\pi |T| /2} .
    \end{align*}  
    \end{proof}
\begin{proposition}\label{thm:gamma_bound_for_large_t}
    Suppose $F \in V^{mg}_{\frac{1}{2}}(N,\chi)$ and let $G=F|W_{\frac{1}{2}} (N)$. \\
    Let $\alpha >  \operatorname{max} \{1+ \{\mu_F^{1},\mu_F^{2},\mu_F^{3} \},1+ \{\mu_G^{1},\mu_G^{2},\mu_G^{3} \}, \frac{1}{4}\}$, $\rho>0$. For $\sigma \in [\frac{1}{2}-\alpha, \alpha]$, 
    we have
    $$
    \left| \frac{\Lambda_{N}(F, \sigma+iT) }{\Gamma(\rho + 1 + \sigma + iT)} \right| = O(|T|^{-\sigma-\frac{1}{2}+\alpha-\rho})
    $$
    as $|T|\to \infty$, uniformly for $\sigma$ in $[\frac{1}{2}-\alpha,\alpha]$.
    \end{proposition}
    \begin{proof} 
    For $\sigma = \alpha,$ we use Lemma \ref{thm:LambdaStirlingBound} and Stirling's formula to obtain as $|T| \to \infty$
    \begin{align}\label{stirlingq}
    \frac{\Lambda_{N}(F, \alpha+iT) }{\Gamma(\rho + 1 + \alpha + iT)} \ll \frac{ e^{- \pi |T| / 2} |T|^{\alpha}}{|T|^{\rho +1 + \alpha- \frac{1}{2}} e^{- \pi |T|/2}} = O(|T|^{-\rho -1/2}). 
    \end{align}
    When $\sigma = \frac{1}{2}-\alpha$, with Theorem \ref{thm:functional-equation} and an application of \eqref{thm:LambdaStirlingBound}, we obtain, as $|T| \to \infty$,
    \begin{align*}
    \frac{\Lambda_{N}(F, \frac{1}{2}-\alpha+iT) }{\Gamma(\rho + 1 + \frac{1}{2}-\alpha + iT)}&=
    i^{\frac{1}{2}}\frac{\Lambda_{N}(G, \alpha-iT) }{\Gamma(\rho + 1 + \alpha - iT)} \cdot
    \frac{\Gamma(\rho + 1 + \alpha - iT)}{\Gamma(\rho + 1 + \frac{1}{2} -\alpha + iT)} \\
    &\ll |T|^{ -\rho -\frac{1}{2} +2 \alpha- \frac{1}{2}}.\end{align*}
    The result then follows from the Phragmen-Lindel\"{o}f Principle.
    \end{proof}

\section{Summation formula in the sesquiharmonic case}\label{sec:summation-formula}
Having established the necessary preliminaries in the previous sections, we are now ready to state and prove our summation formula for the half integral weight sesquiharmonic maass forms. 
As in \cite{bdgrt}, we set
\begin{align}\label{eq:g-rho-def}
g_{\rho}(n,r) := \frac{2 \pi i}{\Gamma(\rho+\frac{1}{2})} \left(1+ \frac{n}{r} \right)^{\rho} \left( \frac{\sqrt{\pi} \Gamma(\rho+\frac{1}{2})}{\Gamma(\rho+1)} - \mathbf{B}_{{\frac{2n}{r+n}}}\left(\frac{1}{2}, \rho+\frac{1}{2}\right)\right).
\end{align}

\begin{theorem} \label{thm:sumformula-sesqui} 
Let $F \in V_{\frac{1}{2}}^{mg}(N,\chi)$ and let $G := F | W(N)$. Let the Fourier coefficients $F$ and $G$ be denoted as in \eqref{eq:fourier}.
Assume that $\mu_F^{1}, \mu_F^{2},\mu_F^{3}, \mu_G^{1},\mu_G^{2},\mu_G^{3} \in \mathbb{R}_{>0}$ such that $a_{F}(n) = \mathcal{O}(n^ {\mu_F^{1}}), b_{F}(n) = \mathcal{O}(n^ {\mu_F^{2}}), c_{F}(n) = \mathcal{O}(n^ {\mu_F^{3}})$ and similarly $a_{G}(n) = \mathcal{O}(n^ {\mu_G^{1}}), b_{G}(n) = \mathcal{O}(n^ {\mu_G^{2}}), c_{G}(n) = \mathcal{O}(n^ {\mu_G^{3}})$. Also let $ \rho_{0} = 2 + 2 \max \{\mu_F^{1},\mu_F^{2},\mu_F^{3},\mu_G^{1},\mu_G^{2},\mu_G^{3}\}$. \\ Then for $\rho > \rho_{0} -\frac{1}{2}$, we have 
\begin{align*}
\frac{1}{\Gamma(\rho +1)}\sum_{n\leq x}{c_{F}(n)(x-n)^{\rho}} + \frac{x^{\rho}}{2{\pi}^{\frac{3}{2}}i} \sum_{n\leq x} {b_{F}(n)g_{\rho}(n,x)}+\frac{x^{\rho}}{2{\pi}i} \sum_{n\leq x} {a_{F}(n)h_{\rho}(n,x)} - x^{\rho} Q_{\rho}(x) \\= -i^{\frac{1}{2}} x^{\frac{\rho}{2} +\frac{1}{4}} \left( \frac{\sqrt{N}}{2{\pi}}\right)^{\rho} \sum_{n=1}^{\infty} {\frac{c_{G}(n)}{n^{\frac{\rho}{2} +\frac{1}{4}}}} J_{{\rho} +\frac{1}{2}} \left(4 {\pi} \sqrt{\frac{nx}{N}}\right) \\
 -i^{\frac{1}{2}} \frac{x^{({\rho}+1)/2}}{\sqrt{\pi}} \left(\frac{\sqrt{N}}{2{\pi}}\right)^{{\rho}-\frac{1}{2}} \sum_{n=1}^{\infty} {\frac{b_{G}(n)}{n^{\frac{\rho}{2}}}} \int_{0}^{1/2} {\frac{u^{\frac{\rho}{2}-1}}{(1-u)^{(1+{\rho})/2}} J_{{\rho} +1} \left( 4 {\pi} \sqrt{\frac{nx(1-u)}{Nu}}\right) du}
\end{align*}
\begin{align}\label{eq:sumformula-sesqui}
-i^{\frac{1}{2}} \frac{x^{({\rho}+1)/2}}{4 \pi^{\frac{3}{2}}}\left( \frac{\sqrt{N}}{2{\pi}}\right)^{{\rho}-\frac{1}{2}} \sum_{n=1}^{\infty} {\frac{a_{G}(n)}{n^{\frac{\rho}{2}}} \int_{0}^{1} \log \left( \frac{t+1}{2t}\right) t^{\frac{\rho}{2} -1} (1-t) ^{-\frac{1}{2}} J_{{\rho} +1} \left( 4 {\pi} \sqrt{\frac{nx}{Nt}}\right) dt }
\end{align}
where
\begin{align}
Q_{\rho}(x)=  \frac{(-d_{F}(0)  +d_{F}(1) \log (\sqrt{N}) )}{\Gamma(\rho+1)} + \frac{i^{\frac{1}{2}} N^{-\frac{1}{4}} d_{G}(2) 2 \pi x }{N^{\frac{1}{2}} \Gamma(\rho+2) } +  \nonumber \\  \frac{(2 \pi x)^{-\frac{1}{2}} ( - d_{F}(2)) N^{-\frac{1}{4}}} {N^{-\frac{1}{4}} \Gamma(\rho+\frac{1}{2})} + \frac{i^{\frac{1}{2}} ( d_{G}(0) -d_{G}(1) \log (\sqrt{N}))  (2 \pi x)^{\frac{1}{2}}}{N^{\frac{1}{4}} \Gamma(\rho+\frac{3}{2})}.
\end{align}
with $g_{\rho}(n,r)$ as in \eqref{eq:g-rho-def} and $h_{\rho}(n,r)$ as given in Proposition $\ref{int-sum of Z(s)}$.
\end{theorem}

\begin{proof}

 We will first prove \eqref{eq:sumformula-sesqui} for $\rho>\rho_0$. 
    Let $\alpha \in \mathbb{R}$ satisfy
    \begin{equation}\label{range} 1+ \operatorname{max} \{\mu_F^{1},\mu_F^{2},\mu_F^{3}, \mu_G^{1},\mu_G^{2},\mu_G^{3} \}< \alpha < \frac{\rho}{2}.
    \end{equation}
    Note that our lower bound on $\rho$ makes such a choice possible. Then Perron's formula gives 
    \begin{align}\label{1}
    \frac{1}{\Gamma(\rho+1)}\sum_{n\leq x} &c_{F}(n)(x-n)^{\rho}=\frac{1}{2\pi i}\int_{(\alpha)}\frac{\Gamma(s)L^{+}(F,s)}{\Gamma(\rho+1+s)}x^{s+\rho}\ ds\nonumber\\
    &=\frac{x^\rho}{2\pi i}\int_{(\alpha)}\frac{\Lambda_{N}(F,s)}{\Gamma(\rho+1+s)}\left(\frac{2\pi x}{\sqrt{N}} \right)^{s} ds  -\frac{x^\rho}{2\pi^{\frac{3}{2}} i}\int_{(\alpha)}\frac{W_{\frac{1}{2}}(s)L^{-}(F,s)}{\Gamma(\rho+1+s)} x^s  ds \nonumber \\
    & -\frac{x^\rho}{2\pi i}\int_{(\alpha)}\frac{Z(s)L^{*}(F,s)}{\Gamma(\rho+1+s)} x^s  ds.
    \end{align}

We will first evaluate the first integral on the right-hand side in \eqref{1} by shifting the line of integration to $\frac{1}{2}-\alpha<0$. By Theorem \ref{thm:functional-equation}, $\Lambda_{N}(F,s)$ has poles at $s=-\frac{1}{2},0,\frac{1}{2}$ and $1$, by Cauchy's Residue Theorem,  we find
    \begin{align}
    \frac{x^{\rho}}{2\pi i}\int_{(\alpha)}\frac{\Lambda_{N}(F,s)}{\Gamma(\rho+1+s)}\left(\frac{2\pi x}{\sqrt{N}} \right)^{s}\ ds= x^{\rho} \sum_{j \in \{-\frac{1}{2},0, \frac{1}{2},1 \}} R_j +\frac{x^{\rho}}{2\pi i}\int_{(\frac{1}{2}-\alpha)}\frac{\Lambda_{N}(F,s)}{\Gamma(\rho+1+s)}\left(\frac{2\pi x}{\sqrt{N}} \right)^{s}\ ds,
    \end{align}
    where $R_j$ is the residue of $\frac{\Lambda_{N}(F,s)}{\Gamma ( \rho + 1 + s)} \left( \frac{2 \pi x}{\sqrt{N}} \right)^s$ at $s=j$. Note that Proposition \ref{thm:gamma_bound_for_large_t} implies that the integrals along two horizontal segments of the usual rectangular contour tend to zero as $|T|$ tends to infinity.
    
Next, we set $Q_{\rho}(x) = R_{0}+R_{1}+R_{-\frac{1}{2}}+R_{\frac{1}{2}}$. Then using Theorem \ref{thm:functional-equation}, we have 
\begin{align*}
Q_{\rho}(x)= \frac{(-d_{F}(0)+d_{F}(1) \log (\sqrt{N}) ) }{\Gamma(\rho+1)} + \frac{i^{\frac{1}{2}} N^{-\frac{1}{4}} (d_{G}(2)2\pi x }{N^{\frac{1}{2}} \Gamma(\rho+2) } + \\ \frac{ (2 \pi x)^{-\frac{1}{2}}   ( - d_{F}(2)) N^{-\frac{1}{4}}}{ N^{-\frac{1}{4}} \Gamma(\rho+\frac{1}{2})} + \frac{i^{\frac{1}{2}} (d_{G}(0) -d_{G}(1) \log (\sqrt{N}))  (2 \pi x)^{\frac{1}{2}}  }{N^{\frac{1}{4}} \Gamma(\rho+\frac{3}{2})}
\end{align*}
Now using the functional equation \eqref{function-equation2}, we obtain
\begin{align*}
 \frac{x^{\rho}}{2\pi i}\int_{(\alpha)}\frac{\Lambda_{N}(F,s)}{\Gamma(\rho+1+s)}\left(\frac{2\pi x}{\sqrt{N}} \right)^{s}\ ds = x^{\rho} Q_{\rho}(x) + i^{\frac{1}{2}} \frac{x^{\rho}}{2\pi i}\int_{(\frac{1}{2}-\alpha)}\frac{\Lambda_{N}(G,\frac{1}{2} - s)}{\Gamma(\rho+1+s)}\left(\frac{2\pi x}{\sqrt{N}} \right)^{s}\ ds \\
=  x^{\rho} Q_{\rho}(x)  -  i^{\frac{1}{2}} \frac{x^{\rho}}{2\pi i} \left(\frac{2\pi x}{\sqrt{N}} \right)^{\frac{1}{2}} \int_{(\alpha)}\frac{\Lambda_{N}(G, s)}{\Gamma(\rho+\frac{3}{2}-s)}\left(\frac{2\pi x}{\sqrt{N}} \right)^{-s}\ ds \\
= x^{\rho} Q_{\rho}(x) -   i^{\frac{1}{2}} x^{\rho}  \left(\frac{2\pi x}{\sqrt{N}} \right)^{\frac{1}{2}} \sum_{n=1}^{\infty} c_{G}(n) \frac{1}{ 2 \pi i}  \int_{(\alpha)}\frac{\Gamma(s)}{\Gamma(\rho+\frac{3}{2}-s)}\left(\frac{4 {\pi}^{2} n x}{N} \right)^{-s}\ ds \\
-   i^{\frac{1}{2}} x^{\rho}  \left(\frac{2\pi x}{\sqrt{N}} \right)^{\frac{1}{2}} \sum_{n=1}^{\infty} b_{G}(n) \frac{1}{ 2 {\pi}^{\frac{3}{2}} i}  \int_{(\alpha)}\frac{W_{\frac{1}{2}}(s)}{\Gamma(\rho+\frac{3}{2}-s)}\left(\frac{4 {\pi}^{2} n x}{N} \right)^{-s}\ ds \\
-   i^{\frac{1}{2}} x^{\rho}  \left(\frac{2\pi x}{\sqrt{N}} \right)^{\frac{1}{2}} \sum_{n=1}^{\infty} a_{G}(n) \frac{1}{ 2 \pi i}  \int_{(\alpha)}\frac{Z(s)}{\Gamma(\rho+\frac{3}{2}-s)}\left(\frac{4 {\pi}^{2} n x}{N} \right)^{-s}\ ds
\end{align*}
To justify the interchange of summation and integration in the first two integrals, one can use the same argument as in \cite{bdgrt}. For the third integral, we decompose the integral and apply Stirling's asymptotic as follows:
\begin{align*} 
 \int_{(\alpha)}\Big|\frac{Z(s)}{\Gamma(\rho+\frac{3}{2}-s)}\left(\frac{4 {\pi}^{2} n x}{N} \right)^{-s}  \Big| ds  \\
&= \left(\frac{4 {\pi}^{2} n x}{N} \right)^{-\alpha}  \left( \int_{-1}^1 + \int_{-\infty}^{-1} + \int_1^{\infty} \right) \Big|\frac{Z(\alpha + it)}{\Gamma(\rho+\frac{3}{2}-\alpha - it)} \Big| dt  \\
\end{align*}
Now from Lemma \ref{thm:holo-Z(s)}, we know that $Z(s)$ is holomorphic on $\Re(s)>0$, and the reciprocal of the gamma function is also continuous (in fact is entire), so the integrand is continous on $\Re(s) >0$. In particular, it is continuous on $[-1,1]$ which is compact. Thus the integrand is bounded and so the integral from $-1$ to $1$ is convergent. 

The other two integrals are absolutely convergent because of Stirling's theorem and Proposition \ref{thm:gamma_bound_for_large_t}:
\begin{align*}
  \left(  \int_{-\infty}^{-1} + \int_1^{\infty} \right) \Big|\frac{Z(\alpha + it)}{\Gamma(\rho+\frac{3}{2}-\alpha - it)} \Big| dt  \ll 2   \int_1^{\infty}  \frac{t^{\alpha - 1} e^{-\pi t /2}}{t^{\rho + \frac{3}{2} - \alpha -\frac{1}{2}} e^{-\pi t /2}}dt 
\end{align*}
which is convergent for $\alpha - 1 - \rho - \frac{3}{2} +\frac{1}{2} + \alpha < -1$, i.e. $\alpha < \frac{\rho +1}{2}$. The interchange is thus justified by the absolute convergence of $\sum a_G(n) n^{-\alpha}$ and Fubini's Theorem.

With this we obtain our summation formula as follows,
 \begin{align}{\label{summation-formula}}
 &\frac{1}{\Gamma(\rho+1)}\sum_{n\leq x} c_{F}(n)(x-n)^{\rho} = x^{\rho} Q_{\rho}(x) \\
&-  i^{\frac{1}{2}} x^{\rho}  \left(\frac{2\pi x}{\sqrt{N}} \right)^{\frac{1}{2}} \sum_{n=1}^{\infty} c_{G}(n) \frac{1}{ 2 \pi i}  \int_{(\alpha)}\frac{\Gamma(s)}{\Gamma(\rho+\frac{3}{2}-s)}\left(\frac{4 {\pi}^{2} n x}{N} \right)^{-s}\ ds \nonumber \\
&-   i^{\frac{1}{2}} x^{\rho}  \left(\frac{2\pi x}{\sqrt{N}} \right)^{\frac{1}{2}} \sum_{n=1}^{\infty} b_{G}(n) \frac{1}{ 2 {\pi}^{\frac{3}{2}} i}  \int_{(\alpha)}\frac{W_{\frac{1}{2}}(s)}{\Gamma(\rho+\frac{3}{2}-s)}\left(\frac{4 {\pi}^{2} n x}{N} \right)^{-s}\ ds \nonumber  \\
&-   i^{\frac{1}{2}} x^{\rho}  \left(\frac{2\pi x}{\sqrt{N}} \right)^{\frac{1}{2}} \sum_{n=1}^{\infty} a_{G}(n) \frac{1}{ 2 \pi i}  \int_{(\alpha)}\frac{Z(s)}{\Gamma(\rho+\frac{3}{2}-s)}\left(\frac{4 {\pi}^{2} n x}{N} \right)^{-s}\ ds -\frac{x^\rho}{2\pi^{\frac{3}{2}} i}\int_{(\alpha)}\frac{W_{\frac{1}{2}}(s)L^{-}(F,s)}{\Gamma(\rho+1+s)} x^s  ds \nonumber \\
  &-\frac{x^\rho}{2\pi i}\int_{(\alpha)}\frac{Z(s)L^{*}(F,s)}{\Gamma(\rho+1+s)} x^s  ds. \nonumber
\end{align}

Now we will give a closed form of each integral in the right hand side of \eqref{summation-formula}. 
For the first two integrals, we have the closed forms from  \cite{bdgrt},
\begin{align*}
\frac{1}{2 \pi i} \sum_{n=1}^{\infty} c_{G}(n)  \int_{(\alpha)} \frac{\Gamma(s)}{\Gamma(\rho+\frac{3}{2}-s)}\left(\frac{4 {\pi}^{2} n x}{N} \right)^{-s} ds 
= \left(\frac{\sqrt{N}}{2 \pi \sqrt{x}}\right)^{\rho +\frac{1}{2}}  \sum_{n=1}^{\infty} \frac{c_{G}(n)}{n^{\frac{\rho}{2}+\frac{1}{4}}} J_{\rho +\frac{1}{2}} \left(4 \pi \sqrt{\frac{nx}{N}}\right)
\end{align*}

In the last step, we employ the following formula from Section 7.3 (23) of \cite{bateman1954tables} with the assumption that $0 < \alpha < \frac{\rho}{2}+\frac{1}{4}$,
\begin{align*}
\frac{1}{2 \pi i}\int_{(\alpha)} \frac{\Gamma(s)}{\Gamma(\rho+\frac{3}{2}-s)} z^{-s} ds= z^{-\left(\frac{\rho}{2} +\frac{1}{4}\right)} J_{\rho+\frac{1}{2}} (2 \sqrt{z})
\end{align*}
For $\alpha < \frac{\rho}{2}$,
 \begin{align*} 
&\sum_{n=1}^{\infty} b_{G}(n) \frac{1}{ 2 {\pi}^{\frac{3}{2}} i}  \int_{(\alpha)}\frac{W_{\frac{1}{2}}(s)}{\Gamma(\rho+\frac{3}{2}-s)}\left(\frac{4 {\pi}^{2} n x}{N} \right)^{-s}\ ds \\
&= \frac{1}{\sqrt{\pi}}\sum_{n=1}^{\infty} b_{G}(n) \int_{0}^{\frac{1}{2}} u^{-1} (1-u)^{-\frac{1}{2}} \left( \frac{4 \pi^{2} xn (1-u)}{Nu} \right)^{\frac{1}{2} - \frac{\rho+1}{2}} J_{{\rho} +1} \left( 4 {\pi} \sqrt{\frac{nx(1-u)}{Nu}}\right) du 
\end{align*}
 
Now for the third integral, by using Lemma \ref{thm:Z-formula-2} and Fubini's theorem, we compute
\begin{align*}   
\frac{1}{ 2 \pi i} & \int_{(\alpha)}\frac{Z(s)}{\Gamma(\rho+\frac{3}{2}-s)}\left(\frac{4 {\pi}^{2} n x}{N} \right)^{-s}\ ds \\
&= \frac{1}{ 4 \pi^{\frac{3}{2}}}  \int_{0}^{\infty} \int_{0}^{\frac{1}{1+2u}} t^{-1}(1-t)^{-\frac{1}{2}} \left(\frac{4 {\pi}^{2} n x}{Nt} \right)^{\frac{1}{2}-\frac{(\rho+1)}{2}} J_{\rho+1} \left( 2 \sqrt{\frac{4 \pi^{2} nx}{Nt}} \right)  \frac{ dt du}{1+u}
\end{align*} 
which holds for $\alpha \leq \frac{\rho}{2}$. Now by applying Fubini's theorem again, we finally obtain 
\begin{align*}
&\frac{1}{ 2 \pi i}\int_{(\alpha)}\frac{Z(s)}{\Gamma(\rho+\frac{3}{2}-s)}\left(\frac{4 {\pi}^{2} n x}{N} \right)^{-s}ds \\
&=  \frac{1}{ 4 \pi^{\frac{3}{2}}} \int_{0}^{1} \log \left(\frac{t+1}{2t} \right) t^{-1}(1-t)^{-\frac{1}{2}} \left(\frac{4 {\pi}^{2} n x}{Nt} \right)^{\frac{1}{2}-\frac{(\rho+1)}{2}} J_{\rho+1} \left( 2 \sqrt{\frac{4 \pi^{2} nx}{Nt}} \right) dt
\end{align*}
The last integral is given by Proposition \ref{thm:finite-sum-formula}. 

Extending the formula from $\rho > \rho_0$ to $\rho > \rho_0 - \frac{1}{2}$ follows exactly as in \cite{bdgrt}. We give a brief sketch of the argument. First, assume that $\rho > \rho_0 - \frac{1}{2}$. By what we have already shown, $\rho + 1 > \rho_0$ implies that the equation in Theorem \ref{thm:sumformula-sesqui} holds with $\rho+1$ replacing $\rho$. Differentiating both sides with respect to $y = \sqrt{x}$ produces the desired relation for $\rho$, after justifying the term-by-term differentiation of the series involved using their absolute convergence on the range $\rho > \rho_0 -\frac{1}{2}$, and using Proposition \ref{thm:finite-sum-formula} to prove the analogues of (4.16) and (4.12) in \cite{bdgrt}. 
\end{proof}
 
\subsection{Asymptotic formula}
We can now deduce the asymptotic result from the above summation formula. 
\begin{theorem}\label{asymptotic-formula}
With the same assumptions as in Theorem \ref{thm:sumformula-sesqui}, for $\rho > \rho_{0} -\frac{1}{2}$ we have 
\begin{align} 
\nonumber \frac{1}{\Gamma(\rho +1)}\sum_{n\leq x}{c_{F}(n)(x-n)^{\rho}} + \frac{x^{\rho}}{2{\pi}^{\frac{3}{2}}i} \sum_{n\leq x} {b_{F}(n)g_{\rho}(n,x)}+\frac{x^{\rho}}{2{\pi}i} \sum_{n\leq x} {a_{F}(n)h_{\rho}(n,x)} &\\ \nonumber =Q_{\rho}(x) + \mathcal{O}(x^{\frac{\rho}{2} +\frac{1}{4}}) 
&\\  =  x^{1+\rho} \frac{i^{\frac{1}{2}} N^{-\frac{1}{4}} (d_{G}(2) 2 \pi)}{N^{\frac{1}{2}} \Gamma(\rho+2)} + \mathcal{O}(x^{\rho +\frac{1}{2}}).
\end{align}
\end{theorem}
\begin{proof}
From \cite{bdgrt}, we have the following bounds
\begin{align}
x^{\frac{\rho}{2} +\frac{1}{4}} \sum_{n=1}^{\infty} \frac{c_{G}(n)}{n^{\frac{\rho}{2}+\frac{1}{4}}} J_{\rho +\frac{1}{2}} \left(4 \pi \sqrt{\frac{nx}{N}}\right) \ll x^{\frac{\rho}{2}} \sum_{n=1}^{\infty} \frac{1}{n^{\frac{\rho}{2} +\frac{1}{2} - \mu_{G}^{3}}}
\end{align}
and 
\begin{align}
\nonumber x^{\frac{\rho+1}{2}}\sum_{n=1}^{\infty} \frac{b_{G}(n)}{n^{\frac{\rho}{2}}} \int_{0}^{\frac{1}{2}} u^{\frac{\rho}{2} -1} (1-u)^{-\frac{(\rho+1)}{2}}   J_{{\rho} +1} \left( 4 {\pi} \sqrt{\frac{nx(1-u)}{Nu}}\right) du \\ \ll x^{\frac{\rho}{2} +\frac{1}{4}} \left( \sum_{n=1}^{\infty} \frac{1}{n^{\frac{\rho}{2} +\frac{1}{4} -  \mu_{G}^{2} }} \right) \int_{0}^{\frac{1}{2}} \frac{u^{\frac{\rho}{2} - \frac{3}{4}}}{(1-u)^{\frac{\rho}{2}+\frac{3}{4}}} du
\end{align}
Now for the other integral, we compute
\begin{align}
\nonumber x^{\rho+\frac{1}{2}} \sum_{n=1}^{\infty} a_{G}(n) \left(\frac{1}{ 4 \pi^{\frac{3}{2}}} \int_{0}^{1} \log \left(\frac{t+1}{2t} \right) t^{-1}(1-t)^{-\frac{1}{2}} \left(\frac{4 {\pi}^{2} n x}{Nt} \right)^{{\frac{1}{2}}-\frac{(\rho+1)}{2}} J_{\rho+1} \left( 2 \sqrt{\frac{4 \pi^{2} nx}{Nt}} \right) dt \right) \\
\ll x^{\frac{\rho}{2} +\frac{1}{4}} \left( \sum_{n=1}^{\infty} \frac{1}{n^{\frac{\rho}{2} +\frac{1}{4} -  \mu_{G}^{1} }} \right)  \int_{0}^{1} \log \left(\frac{t+1}{2t} \right) t^{\frac{\rho}{2}{- \frac{3}{4}}}(1-t)^{-\frac{1}{2}} dt 
\end{align}
which follows from the bound $J_{v}(x) \ll x^{-\frac{1}{2}}$. 

Now, Theorem \ref{thm:sumformula-sesqui}, we deduce for $\rho > \rho_{0} -\frac{1}{2}$,
\begin{align*}
\frac{1}{\Gamma(\rho +1)}\sum_{n\leq x}{c_{F}(n)(x-n)^{\rho}} + \frac{x^{\rho}}{2{\pi}^{\frac{3}{2}}i} \sum_{n\leq x} {b_{F}(n)g_{\rho}(n,x)}+\frac{x^{\rho}}{2{\pi}i} \sum_{n\leq x} {a_{F}(n)h_{\rho}(n,x)} \\= Q_{\rho}(x)  + \mathcal{O}(x^{\frac{\rho}{2} +\frac{1}{4}}) 
\end{align*}
Now since $\rho+1 > \rho +\frac{1}{2} > \frac{\rho}{2} +\frac{1}{2} >   \frac{\rho}{2} +\frac{1}{4}$ and $Q_{\rho}(x)$ is linear combination of $x^{\rho}, x^{\rho - \frac{1}{2}}, x^{\rho + \frac{1}{2}}, x^{\rho + 1}$, we obtain \eqref{asymptotic-formula}.
\end{proof}

\section{Real quadratic class numbers}\label{sec:example}
 
\subsection{Sesquiharmonic Maass forms related to real quadratic fields} 
As mentioned in the introduction, Theorem \ref{thm:asymptotic-formula-example} is obtained by applying Theorem \ref{thm:sumformula-sesqui} to a sesquiharmonic Maass form introduced in \cite{DukeImamogluToth}. To define it, we need a few definitions. First, we recall that for positive indices, the Hurwitz class numbers count $SL_{2}(\mathbb{Z})$ classes of binary quadratic forms inversely weighted by stabilizer size:
\begin{align}\label{eq:Hurwitz-defn}
H(n) := \sum_{Q \in SL_{2}(\mathbb{Z}) \backslash \mathcal{Q}_{-n}} { \frac{2}{|Stab(Q)|}},
\end{align}
where for any nonzero discriminant $d$, $\mathcal{Q}_{d}$ is the set of positive definite binary quadratic forms of discriminant $d$. We also use the convention $H(0) = -\frac{1}{12}$. 

Now for a fundamental discriminant $d >1$, define 
\begin{align}\label{eq:exp-trace formula}
\mathrm{Tr_{d}(1)} =  \frac{ h(d) \log \epsilon_d }{\pi \sqrt{d}} 
\end{align}
where $h(D)$ is the narrow class number and $\epsilon_D$ is the smallest unit greater than 1 of norm 1 in the quadratic order of conductor $D$. 

\begin{theorem}[Theorem 4 \cite{DukeImamogluToth}]\label{thm:DIT-form}
With $\mathrm{Tr_{d}(1)}$ defined appropriately for nonfundamental $d>1$, the function $\widehat{Z}_{+}(\tau)$ with Fourier expansion given by
\begin{align}\label{eq:exp-Z(tau)}
    \widehat{Z}_{+}(\tau) = &\sum_{d > 0}  \mathrm{Tr_{d}(1)} q^d + \sum_{d < 0} \frac{H(|d|)}{\sqrt{|d|}} \beta(4 |d| y) q^d  
    + \sum_{n \geq 1} 2\alpha(4 n^2 y) q^{n^2} - \frac{ \log y}{4 \pi} + \frac{\sqrt{y}}{3}
\end{align}
is a sesquiharmonic Maass form of weight $\frac{1}{2}$ for $\Gamma_0(4)$. 
\end{theorem}
\begin{remark} 
For nonsquare $d >1$, (1.5) of \cite{AAS} shows that
\begin{align*}
\mathrm{Tr_{d}(1)} =  \frac{1}{\pi \sqrt{d}} \sum_{\substack{\ell^2 | d \\ d/\ell^2 \equiv 0,1 \pmod{4}} } h(d/\ell^2) \log \epsilon_{d/\ell^2} 
\end{align*}
The coefficients of square index of this function were computed in \cite{beckwith2024modular}:
$$
\mathrm{Tr_{d^2}(1)} = 
\frac{2}{3\pi} \left(\gamma - 2\frac{\zeta'(2)}{\zeta(2)} - \log(4) + \frac{1}{2}\log(\pi) - \sum_{\substack{f|d \\ r | f \\ (2,fr)=1}} \frac{\mu(f) \log (f r^2 )}{fr} \right).
$$ 
\end{remark}
\begin{remark}
Higher level analogs of $\widehat{Z}_{+}(\tau)$ were defined in \cite{BeckwithMono-I}. In particular, for every square free odd $N$, Theorem 1.3 of that paper gives the Fourier expansion of the analogue of $ \widehat{Z}_{+}(\tau)$ for the group $\Gamma_0(4N)$.
\end{remark}

 \subsection{The $W(4)$ Operator}
To compute the  $\widehat{Z}_{+}(\tau)|W(4)$, we use the equation (2.3) in \cite{DukeImamogluToth} which says that if $f$ is a smooth function of weight $k \in \mathbb{Z} + 1/2$ on $\Gamma_{0}(4)$ which has the following Fourier expansion
\begin{align*}
f(\tau)= \sum_{n} {a(n;y) e(nx)}
\end{align*}
where $e(z)= e^{2 \pi iz}$, then 
\begin{align}\label{eq:exp-1/4tau}
f\left(-\frac{1}{4 \tau}\right)= \left(\frac{2\tau}{i}\right)^{k} \alpha f^{e}(\tau)
\end{align}
where $f^{e}(\tau) = \sum_{2|n} a(n;\frac{y}{4}) e(\frac{nx}{4})$ and $\alpha = (-1)^{\lfloor \frac{2k+1}{4} \rfloor} 2^{-k +\frac{1}{2}}$.\\
\begin{lemma}
Using the Frick involution we have 
\begin{align}\label{Z(tau)|W(4)}
\widehat{Z}_{+}(\tau)|W(4) =   &\sum_{d > 0} \frac{(1-i)}{\sqrt{2}} \mathrm{Tr_{d}(1)} q^{d} + \sum_{d < 0} \frac{(1-i)}{2\sqrt{2}} \frac{H(|4d|)}{\sqrt{|d|}} \beta( 4|d| y)  q^{d}  \\ \nonumber &+  \sum_{n > 0} \sqrt{2}(1-i)\alpha(4 n^2 y)  q^{n^{2}} - \frac{(1-i)}{\sqrt{2}} \frac{1}{4 \pi} \log y + \frac{(1-i)}{\sqrt{2}} \frac{\sqrt{y}}{6} \\ \nonumber & + \frac{(1-i)}{\sqrt{2}}\frac{1}{2 \pi} \log 2.  
\end{align}
\end{lemma}
\begin{proof}
 Setting $k=\frac{1}{2}$ and using \eqref{eq:exp-1/4tau}, we have 
\begin{align}\label{eq:exp-Z(tau)W4}
 \widehat{Z}_{+}(\tau)|W(4)&= \nonumber 2^{-\frac{1}{2}} \tau^{-\frac{1}{2}} \widehat{Z}_{+}\left(-\frac{1}{4\tau}\right)\\
&= \nonumber 2^{-\frac{1}{2}} \tau^{-\frac{1}{2}} 2^{\frac{1}{2}} \tau^{\frac{1}{2}} i^{-\frac{1}{2}} \widehat{Z}_{+}^{e}(\tau)\\
&= i^{-\frac{1}{2}} \widehat{Z}_{+}^{e}(\tau).
\end{align}
Let 
\begin{align*}
&\sum_{d > 0}  \mathrm{Tr_{d}(1)}  q^d + \sum_{d < 0} \frac{H(|d|)}{\sqrt{|d|}} \beta(4 |d| y) q^d +\frac{\sqrt{y}}{3} = \sum_{d} a(d) \mathcal{W}_{d}(y) e(dx)
\end{align*} 
  where 
  \begin{align*}
\mathcal{W}_{d}(y) := \begin{cases} 
  	   |d|^{-\frac{1}{2}} \beta(4 |d| y) e^{-2 \pi dy} , & \text{if} \ d <0 \\
    	  -4y^{\frac{1}{2}}, & \text{if } d=0 \\
    	   e^{-2 \pi dy} & \text{if } d>0
	\end{cases}
\end{align*}
and 
\begin{align*}
a(d) = \begin{cases}
H(|d|) & \text{if} \ d <0\\
-\frac{1}{12} & \text{if} \ d=0\\
\mathrm{Tr_{d}(1)} & \text{if} \ d>0.
\end{cases}
\end{align*}
On the other hand, let 
\begin{align*}
\sum_{n \geq 1} 2\alpha(4 n^2 y) q^{n^2}- \frac{ \log y}{4 \pi}= \sum_{n \ge 0} b(n) \mathcal{M}_{n}(y) e(dx) 
\end{align*}
where 
  \begin{align*}
\mathcal{M}_{n}(y) := \begin{cases} 
  	    \log y , & \text{if } n=0 \\
    	  \alpha( 4 n^{2} y) e^{-2 \pi n^{2} y} & \text{if } n>0
	\end{cases}
\end{align*}
and 
\begin{align*}
b(n) = \begin{cases}
2 & \text{if} \ n>0 \ \text{and}\ \Box \\
-\frac{1}{4 \pi} & \text{if} \ n=0\\
\end{cases}
\end{align*}

Therefore
\begin{align*}
\widehat{Z}_{+}(\tau)= \sum_{d} a(d;y) e(dx)+\sum_{n} b(n;y)  e(nx)  
\end{align*} 
where $ a(d;y)= a(d) \mathcal{W}_{d}(y)$ and $b(n;y)= b(n) \mathcal{M}_{n}(y)$.
Then 
\begin{align}\label{eq:exp-Z-even}
\widehat{Z}_{+}^{e}(\tau)= \sum_{ \substack{ d \equiv 0 (2)}} a\left(d;\frac{y}{4}\right) e\left(\frac{dx}{4}\right)+\sum_{ \substack{n \ge 0\\ n \equiv 0 (2)}} b\left(n;\frac{y}{4}\right)  e\left(\frac{nx}{4}\right)  
\end{align}
where
\begin{align}\label{eq:exp-a-W}
a\left(d;\frac{y}{4}\right)= a(d)\mathcal{W}_{d} \left(\frac{y}{4}\right) := \begin{cases} 
  	   H(|d|) |d|^{-\frac{1}{2}} \beta( |d| y) e^{- \frac{2 \pi dy}{4}} , & \text{if} \ d <0 \\
    	  \frac{1}{3}\left(\frac{y}{4}\right)^{\frac{1}{2}}, & \text{if } d=0 \\
    	 \mathrm{Tr_{d}(1)} e^{- \frac{2 \pi dy}{4}} & \text{if } d>0
	\end{cases}
\end{align}
and
 \begin{align}\label{eq:exp-b-M}
 b\left(n;\frac{y}{4}\right)= b(n)\mathcal{M}_{n} \left(\frac{y}{4}\right) := \begin{cases} 
  	   - \frac{1}{4 \pi}\log \left(\frac{y}{4}\right)   , & \text{if } n=0 \\
    	  2 \alpha(n^{2} y)  e^{- \frac{2 \pi n^{2}y}{4}} & \text{if } n>0
	\end{cases}
\end{align}
Now substituting \eqref{eq:exp-a-W} and \eqref{eq:exp-b-M} in \eqref{eq:exp-Z-even}, we get 
\begin{align*}
 \widehat{Z}_{+}^{e}(\tau) 
  = &\sum_{\substack{d > 0\\ d \equiv 0 (2)}}  \mathrm{Tr_{d}(1)}  q^{d/4} + \sum_{\substack{d < 0\\ d \equiv 0 (2)}} \frac{H(|d|)}{\sqrt{|d|}} \beta( |d| y)  q^{d/4}  + \sum_{\substack{n > 0\\ n \equiv 0 (2)}} 2\alpha( n^2 y)  q^{n^{2}/4} - \frac{1}{4 \pi} \log \left(\frac{y}{4}\right) + \frac{\sqrt{y}}{6}   
\end{align*}
Therefore, 
\begin{align}\label{eq:exp-Z-even}
\widehat{Z}_{+}^{e}(\tau)= &\sum_{d > 0}  \mathrm{Tr_{4d}(1)}q^{d} + \sum_{d < 0} \frac{H(|4d|)}{\sqrt{|4d|}} \beta( 4|d| y)  q^{d}  + \sum_{n > 0} 2\alpha(4 n^2 y)  q^{n^{2}} - \frac{1}{4 \pi} \log \left(\frac{y}{4}\right) + \frac{\sqrt{y}}{6}  
\end{align}
Now by substituting $i^{-\frac{1}{2}}= \frac{(1-i)}{\sqrt{2}}$ and \eqref{eq:exp-Z-even} in \eqref{eq:exp-Z(tau)W4} we finally obtain the required expression.
\end{proof}   
 
\subsection{Summation formula for $\mathrm{Tr_{n}(1)}$}
Now to derive the summation formula for $\mathrm{Tr_{n}(1)}$, we first have the following by comparing the coefficients of \eqref{eq:exp-Z(tau)} with \eqref{eq:fourier},

\begin{align*}
c_{F}(n) &=  \mathrm{Tr_{n}(1)}, \quad b_{F}(n) =  \frac{H(|n|)}{\sqrt{|n|}} \\
a_{F}(n)&=  2  \hspace{1cm} \text{if} \ n = \Box \ \text{and} \ 0 \ \text{otherwise}\\
d_{F}(2) &=  \frac{1}{3}, \quad d_{F}(1) =  -\frac{1}{4 \pi}, \quad d_{F}(0) = 0 
\end{align*}
  and similarly, by comparing the coefficients of \eqref{Z(tau)|W(4)} with the Fourier expansion of $G(\tau)$, we get
\begin{align*}
c_{G}(n) &=  \frac{(1-i)}{\sqrt{2}} \mathrm{Tr_{4n}(1)}, \quad b_{G}(n) = \frac{(1-i)}{2\sqrt{2}} \frac{H(|4d|)}{\sqrt{|d|}} \\
a_{G}(n)&= \sqrt{2}(1-i) \hspace{1cm}  \text{if} \ n = \Box \ \text{and} \ 0 \ \text{otherwise}\\
d_{G}(2) &= \frac{(1-i)}{6\sqrt{2}}, \quad d_{G}(1) = -\frac{(1-i)}{\sqrt{2}} \frac{1}{4 \pi}, \quad d_{G}(0) = \frac{(1-i)}{\sqrt{2}} \frac{1}{2 \pi} \log 2.
 \end{align*}
Moreover, for any $\epsilon >0$, we have $\mu_F^{1} =\mu_F^{2}= \mu_F^{3}=\epsilon$ by the growth of $H(n)$ and $\mathrm{Tr_{n}(1)}$ (\cite{DukeImamogluToth}) and from \eqref{Z(tau)|W(4)}, we get $\mu_G^{1} =\mu_G^{2}= \mu_G^{3}=\epsilon$.
With all these conditions, as an application of Theorem \ref{thm:sumformula-sesqui}, we deduce the following.
\begin{theorem}
Let $H(n)$ be the Hurwitz class number and $\mathrm{Tr_{n}(1)}$ as defined in \eqref{eq:exp-trace formula}.  For $\rho > \frac{3}{2}$, we have 
\begin{align}
\nonumber \frac{1}{\Gamma(\rho +1)} \sum_{n\leq x}{\mathrm{Tr_{n}(1)}(x-n)^{\rho}} + \frac{x^{\rho}}{2{\pi}^{\frac{3}{2}}i} \sum_{n\leq x} {\frac{H(|n|)}{\sqrt{|n|}}g_{\rho}(n,x)}+\frac{x^{\rho}}{2{\pi}i} \sum_{n\leq \sqrt{x}} {2h_{\rho}(n^{2},x)} \\ \nonumber
-x^{\rho} \left( \frac{ - \log 2}{4 \pi \Gamma(\rho +1)} + \frac{\pi x}{6 \sqrt{2} \Gamma(\rho+2)}- \frac{1}{3 \sqrt{2 \pi x} \Gamma(\rho +\frac{1}{2})} + \frac{3\sqrt{\pi x}(\log 2)}{4 \pi \Gamma(\rho + \frac{3}{2})}\right) \\ \nonumber
= -\frac{ x^{\frac{\rho}{2} +\frac{1}{4}}}{ \pi^{\rho}} \sum_{n=1}^{\infty} {\frac{\mathrm{Tr_{4n}(1)}}{n^{\frac{\rho}{2} +\frac{1}{4}}}} J_{{\rho} +\frac{1}{2}} \left(2 {\pi} \sqrt{nx}\right) \\ \nonumber
 -\frac{x^{({\rho}+1)/2}}{2 \pi^{\rho}}  \sum_{n=1}^{\infty} {\frac{H(4n)}{n^{\frac{\rho}{2} +\frac{1}{2}}}} \int_{0}^{1/2} {\frac{u^{\frac{\rho}{2}-1}}{(1-u)^{(1+{\rho})/2}} J_{{\rho} +1} \left( 2 {\pi} \sqrt{\frac{nx(1-u)}{u}}\right) du} \\ -\frac{x^{({\rho}+1)/2}}{2 \pi^{\rho +1}} \sum_{n=1}^{\infty} {\frac{1}{n^\rho} \int_{0}^{1} \log \left( \frac{t+1}{2t}\right) t^{\frac{\rho}{2} -1} (1-t) ^{-\frac{1}{2}} J_{{\rho} +1} \left( 2{\pi} n \sqrt{\frac{x}{t}}\right) dt }
\end{align}
where $g_{\rho}(n,x)$ and $h_{\rho}(n^{2},x)$ are as in Theorem \ref{thm:sumformula-sesqui}.
\end{theorem}

\subsection{Proof of Theorem \ref{thm:asymptotic-formula-example}} 
From Theorem \ref{asymptotic-formula}, we have 
\begin{align*} 
\frac{1}{\Gamma(\rho +1)} \sum_{n\leq x}{\mathrm{Tr_{n}(1)}(x-n)^{\rho}} + \frac{x^{\rho}}{2{\pi}^{\frac{3}{2}}i} \sum_{n\leq x} {\frac{H(|n|)}{\sqrt{|n|}}g_{\rho}(n,x)}+\frac{x^{\rho}}{2{\pi}i} \sum_{n\leq \sqrt{x}} {2h_{\rho}(n^{2},x)} 
&\\  = \frac{ \pi x^{\rho+1}}{6 \sqrt{2} \Gamma(\rho+2)}  + \mathcal{O}(x^{\rho +\frac{1}{2}}).
\end{align*}
Now, by Lemma \ref{thm:finite-sum-bound}, for any $\epsilon >0$, the sum $\left|\sum_{n\leq \sqrt{x}} {2h_{\rho}(n^{2},x)}\right| \leq C x^{\alpha}$ for $\alpha = \frac{1}{2}+\epsilon$. Therefore we have 
\begin{align}\label{eq:asympt-trace-Hurwitz}
\frac{1}{\Gamma(\rho +1)} \sum_{n\leq x}{\mathrm{Tr_{n}(1)}(x-n)^{\rho}}+ \frac{x^{\rho}}{2{\pi}^{\frac{3}{2}}i} \sum_{n\leq x} {\frac{H(|n|)}{\sqrt{|n|}}g_{\rho}(n,x)} &=  \frac{ \pi x^{\rho+1}}{6 \sqrt{2} \Gamma(\rho+2)} +\mathcal{O}(x^{\rho+\frac{1}{2}+\epsilon}).
\end{align}

We can simplify the sum involving $g_{\rho}(n,x)$ as follows. 
Recall that
$$
g_{\rho}(n,x) := \frac{2 \pi i}{\Gamma(\rho+\frac{1}{2})} \left(1+ \frac{n}{x} \right)^{\rho} \left( \frac{\sqrt{\pi} \Gamma(\rho+\frac{1}{2})}{\Gamma(\rho+1)} - \mathbf{B}_ {{\frac{2n}{x+n}}}\left(\frac{1}{2}, \rho+\frac{1}{2}\right)\right).
$$
Using \protect{\cite[(8.17.4)]{NIST:DLMF}}, we have
\begin{align*}
\left(1+ \frac{n}x \right)^{\rho} & \left( 1 - \frac{\Gamma(\rho +1)}{\Gamma(\rho+\frac{1}{2}) \sqrt{\pi}} \mathbf{B}_{{\frac{2n}{x+n}}}\left(\frac{1}{2}, \rho+\frac{1}{2}\right)\right)  \\
&= \left(1+ \frac{n}x \right)^{\rho}  I_{\frac{x-n}{x+n}} ( \rho+\frac{1}{2}, \frac{1}{2} ) \\
&= (1- n/x)^{\rho} \frac{ ((x/n) -1)^{\frac{1}{2}}  }{\sqrt{2} (\rho + \frac{1}{2})} { }_2F_1 \left(1, \frac{1}{2}, \rho + \frac{3}{2}; \frac{1}{2} ( 1 - \frac{x}{n}) \right)
\end{align*}

This completes the proof of Theorem \ref{thm:asymptotic-formula-example}.

\begin{example}
Let $\rho =2$. Then Theorem \ref{thm:asymptotic-formula-example} can be written as
\begin{align*}
\sum_{n\leq x}{\mathrm{Tr_{n}(1)}(1-(n/x))^2}+  \sum_{1 \le n\leq x} \frac{H(n)}{\sqrt{n}}  \left(1+ \frac{n}{x} \right)^2 \left( 1 - \frac{8}{3 \pi} \mathbf{B}_{{\frac{2n}{2+n}}}\left(\frac{1}{2},  \frac{5}{2}\right) \right)= \frac{\pi x}{18 \sqrt{2}} + \mathcal{O}(x^{\frac{1}{2}+\epsilon}).
\end{align*}
We compute the two sums divided by $x$ for several large values of $x$ in the following table:
\begin{center}
\begin{tabular}{  |c| c| c| }
\hline
 $x$ & $S_1(x)/x$   & $S_2(x)/x$ \\ 
 \hline 
$10^2$ & 0.3233 & 0.1873  \\  
$10^3$ & 0.495772 & 0.25618   \\
$10^4$ &0.60973 & 0.2809  \\
$2 \cdot 10^4$ & 0.63162  & 0.2844 \\
$10^5$ & 0.670 & 0.2891 \\
\hline
\end{tabular}
\end{center}
where
\begin{align*}
S_1(x) &:= \frac{ 18 \sqrt{2}}{\pi} \sum_{n\leq x}{\mathrm{Tr_{n}(1)}(1-(n/x))^2}, \\
S_2(x)&:= \frac{ 18 \sqrt{2}}{\pi}   \sum_{1 \le n\leq x} \frac{H(n)}{\sqrt{n}}  \left(1+ \frac{n}{x} \right)^2 \left( 1 - \frac{8}{3 \pi} \mathbf{B}_{\frac{2n}{x+n}}\left(\frac{1}{2},  \frac{5}{2}\right) \right)
\end{align*}
The sum of the two columns shows slow convergence toward 1, as expected.
\end{example}

\begin{corollary}\label{thm:main-result-Tr_n(1)}
For $\rho > \frac{3}{2}$ and any $\delta >0$, 
\begin{align}
 \sum_{n\leq x}{\mathrm{Tr_{n}(1)}(x-n)^{\rho}} = o(x^{\rho+1+\delta})
\end{align}
where implied constant depends on $\rho$.
\end{corollary}

\begin{proof}
By Lemma 4.4 in \cite{bdgrt}, since $\frac{H(|n|)}{\sqrt{|n|}} \ll n^{\epsilon}$, we have $\sum_{n\leq x} {\frac{H(|n|)}{\sqrt{|n|}}g_{\rho}(n,x)} \ll \mathcal{O}(x^{\alpha}) $ \ as $x$ tends to infinity for $\alpha> 1$. 
Applying this bound to the formula \eqref{eq:asympt-trace-Hurwitz} yields that for any $\epsilon >0$, 
\begin{align*}
\frac{1}{\Gamma(\rho +1)} \sum_{n\leq x}{\mathrm{Tr_{n}(1)}(x-n)^{\rho}} =\mathcal{O}(x^{\rho+1+\epsilon}).
\end{align*}
\end{proof}

 \bibliographystyle{alpha}
\bibliography{references.bib} 
\end{document}